\documentclass[12pt]{amsart}
\usepackage{}
\usepackage[all]{xy}
\usepackage{amsmath}
\usepackage{amsfonts}
\usepackage{amssymb}
\usepackage[all]{xy}

\usepackage{bbding}
\usepackage{txfonts}
\usepackage{amscd}

\usepackage[shortlabels]{enumitem}
\usepackage{ifpdf}
\ifpdf
 \usepackage[colorlinks,final,backref=page,hyperindex]{hyperref}
\else
  \usepackage[colorlinks,final,backref=page,hyperindex,hypertex]{hyperref}
\fi
\usepackage{tikz}
\usepackage[active]{srcltx}

\makeatletter

\newtheorem{thm}{Theorem}[section]
\newtheorem{lem}[thm]{Lemma}
\newtheorem{cor}[thm]{Corollary}
\newtheorem{pro}[thm]{Proposition}
\newtheorem{ex}[thm]{Example}
\newtheorem{rmk}[thm]{Remark}
\newtheorem{defi}[thm]{Definition}

\newcommand {\emptycomment}[1]{}
\newcommand {\yh}[1]{{\marginpar{*}\scriptsize\textcolor{purple}{yh: #1}}}
\newcommand {\tr}[1]{{\marginpar{*}\scriptsize\textcolor{red}{tr: #1}}}

\newcommand{\lon }{\,\rightarrow\,}
\newcommand{\be }{\begin{equation}}
\newcommand{\ee }{\end{equation}}

\newcommand{\mg}{\mathfrak g}

\newcommand{\g}{\mathfrak g}
\newcommand{\h}{\mathfrak h}
\newcommand{\al}{\alpha}
\newcommand{\WH}{\mathrm{WH}}

\newcommand{\huaB}{\mathcal{B}}
\newcommand{\huaS}{\mathcal{S}}
\newcommand{\huaA}{\mathcal{A}}
\newcommand{\huaL}{\mathcal{L}}

\newcommand{\huaF}{\mathcal{F}}
\newcommand{\huaG}{\mathcal{G}}

\newcommand{\Z}{\mathbb{Z}}

\newcommand{\huaU}{\mathcal{U}}

\newcommand{\huaW}{\mathcal{W}}

\newcommand{\huaC}{{\mathcal{C}}}

\newcommand{\huaH}{\mathcal{H}}

\newcommand{\huaZ}{\mathcal{Z}}

\newcommand{\frkh}{\mathfrak h}

\newcommand{\frkk}{\mathfrak k}

\newcommand{\frkD}{\mathfrak D}

\newcommand{\frkH}{\mathfrak H}

\newcommand{\frkX}{\mathfrak X}

\newcommand{\NR}{\mathsf{NR}}

\newcommand{\half}{\frac{1}{2}}

\newcommand{\Courant}[1]{\left\llbracket  #1\right\rrbracket }

\newcommand{\Id}{\rm{Id}}

\newcommand{\br}[1]{   [ \cdot,    \cdot  ]   }

\newcommand{\M}{\mathrm{M}}

\newcommand{\Hom}{\mathsf{Hom}}

\newcommand{\Rep}{\mathsf{Rep}}
\newcommand{\ARep}{\mathsf{ARep}}
\newcommand{\WRep}{\mathsf{WRep}}

\newcommand{\Der}{\mathsf{Der}}
\newcommand{\Lie}{\mathrm{Lie}}

\newcommand{\ob}{\mathrm{Ob}}

\newcommand{\Nij}{\mathsf{Nij}}

\newcommand{\diver}{\mathrm{div}}
\newcommand{\gl}{\mathfrak {gl}}
\newcommand{\sln}{\mathfrak {sl}}

\newcommand{\Mod}{\mathrm{Mod}}

\newcommand{\ad}{\mathsf{ad}}
\newcommand{\pr}{\mathrm{Pr}}

\newcommand{\Img}{\mathrm{Im}}

\newcommand{\K}{\mathbb{K}}

\def\D{\mathfrak{D}}

\def\L{\mathcal{L}}
\def\M{\mathcal{M}}

\def\G{\mathcal{G}}

\def\ot{\otimes}

\newcommand{\C}{\mathbb{C}}

\begin{document}

\title[Actions of monoidal categories and crossed homomorphisms]{Actions of monoidal categories and representations of Cartan type Lie algebras}

\author[Yufeng Pei]{Yufeng Pei}
\address{Department of Mathematics, Shanghai Normal University, Guilin Road 100,
Shanghai 200234, China}\email{pei@shnu.edu.cn}

\author{Yunhe Sheng}
\address{Department of Mathematics, Jilin University, Changchun 130012, Jilin, China}
\email{shengyh@jlu.edu.cn}

\author{Rong Tang}
\address{Department of Mathematics, Jilin University, Changchun 130012, Jilin, China}
\email{tangrong16@mails.jlu.edu.cn}

\author{Kaiming Zhao}
\address{Department of Mathematics, Wilfrid
Laurier University, Waterloo, ON, Canada N2L 3C5,  and School of Mathematical Sciences, Hebei Normal University, Shijiazhuang 050024, Hebei, China}
\email{kzhao@wlu.ca}

\vspace{-5mm}

\begin{abstract}
 Using crossed homomorphisms, we show that the category of weak representations (resp. admissible representations) of Lie-Rinehart algebras (resp. Leibniz pairs) is a left module category over the monoidal category of representations of Lie algebras. In particular, the corresponding bifunctor  of monoidal categories
is established  to give new weak representations (resp. admissible representations) of  Lie-Rinehart algebras   (resp. Leibniz pairs). This generalizes and unifies various existing constructions of representations of  many Lie algebras by using this new bifunctor.   We  construct some   crossed homomorphisms in different situations and use our  actions of monoidal categories to recover some known constructions of representations of various Lie algebras, also to obtain new representations for generalized Witt algebras  and their Lie subalgebras. The cohomology theory of crossed homomorphisms between Lie algebras is introduced and used to  study linear deformations of crossed homomorphisms.
\end{abstract}



\maketitle
\vspace{-10mm}
\tableofcontents

\allowdisplaybreaks

\section{Introduction}

This paper aims to give a conceptual approach to unify various constructions of representations of certain Lie algebras and construct new representations of some Lie algebras using crossed homomorphisms, Lie-Rinehart algebras and Leibniz pairs.

\subsection{Representations of Cartan type Lie algebras} The representation theory of Lie algebras is of great importance due to its own overall completeness, and applications in mathematics and mathematical physics. The Cartan type Lie algebras, originally introduced and studied by Cartan,    consist of four classes of  infinite-dimensional simple Lie algebras of vector fields with   formal power series coefficients:   the Witt algebras, the divergence-free algebras, the Hamiltonian algebras, and the contact algebras.
 The representation theory of  Cartan type Lie algebras  was first studied by Rudakov \cite{Ru1,Ru2}. He showed that irreducible  continuous representations can be described explicitly as induced representations or quotients of induced representations. Later Shen \cite{Sh} studied graded modules of graded Lie algebras of Cartan type with   polynomial coefficients of positive characteristic. Larsson constructed a class of representations for the Witt algebras with   Laurent polynomial coefficients \cite{L3}.
More precisely,    Shen's   modules, called  mixed product, were constructed by certain monomorphism; while Larsson's   modules, named   conformal fields,  came from physics background.
 Many other authors have contributed very much to the theory along these approaches for the last few decades. In particular,    irreducible 	modules with finite-dimensional weight spaces over the Virasoro algebra  (universal central extension of   the Lie algebra   $W_1$  of vector fields on a circle) was classified by Mathieu in \cite{Mathieu}, while Billig and Futorny recently gave the classification of irreducible modules over the Witt algebras $W_n \, \, (n\geq2)$ with finite-dimensional weight spaces \cite{BF0}.
 Note that intrinsically there is a functor from the category of finite-dimensional irreducible  representations of finite dimensional simple Lie algebras to the category of representations of Cartan type Lie algebras among these works. Actually there should be some essential part that applies to all those constructions (even more) of complicated modules over some classes of Lie algebras (not only Cartan type Lie algebras) as a whole regardless of any specific feature exhibited in each particular case. From this point of view   it is of no surprise that earlier results in this direction due to many authors are fragments of the general theory. We  find a unifying  conceptual approach   generalizing Shen's construction.  This is one of the main purposes of the paper.

\subsection{Representations of  Lie-Rinehart algebras and Leibniz pairs}

Note that the above mentioned Cartan type Lie algebras are either Lie-Rinehart algebras, or Leibniz pairs.

Lie-Rinehart algebras, which was originally studied by Rinehart in \cite{Ri} in 1963,   arose from a wide variety of constructions in differential geometry, and  they have been introduced repeatedly into many areas under different terminologies, e.g. Lie pseudoalgebras. Lie-Rinehart algebras are the underlying structures of Lie algebroids. See  \cite{Ma} and references therein for more details. A Lie-Rinehart algebra is a quadruple $({A},\L,[\cdot,\cdot]_\L,\alpha)$, where $A$ is a commutative associative algebra, $\L$ is an $A$-module, $[\cdot,\cdot]_\L$ is a Lie bracket on $\L$ and $\alpha:\L\to\Der_\K(A)$ is an $A$-module homomorphism with some compatibility conditions involving the Lie brackets.  Lie-Rinehart algebras have been further investigated in many aspects \cite{CLP1,H1,H2,H5,Ma2,MM}. In particular, Rinehart constructed the  universal enveloping algebra of a Lie-Rinehart algebra \cite{Ri}. Huebschmann gave an alternative construction of the  universal enveloping algebra $\huaU(A, \L)$ of a Lie-Rinehart algebra $({A},\L,[\cdot,\cdot]_\L,\alpha)$ via the smash product, namely
	$
	\huaU(A, \L) =(A\# U(\L))/J,
	$
where $J$ is a certain two-sided ideal in $A\# U(\L)$, and showed that there is a one-one correspondence between representations of a Lie-Rinehart algebra and
	representations of its universal enveloping algebra \cite{H1}.  Representations of  Lie-Rinehart algebras are deeply related to the theory of $\mathcal{D}$-modules \cite{Pen}, which are   modules over the algebra $\mathcal{D}$ of linear differential operators on a manifold. Since the algebra  $\mathcal{D}$  is the universal enveloping algebra of the Lie-Rinehart algebra of vector fields, a $\mathcal{D}$-module is the same as  a module with a representation of the Lie-Rinehart algebra of vector fields. We introduce the notion of a weak representation of a Lie Rinehart algebra. The adjoint action is naturally a weak representation of a Lie-Rinehart algebra on itself. There is a one-to-one correspondence between weak representations of a Lie-Rinehart algebra and	representations of the smash product $A\# U(\L)$.

The notion of a Leibniz pair was originally introduced by Flato-Gerstenhaber-Voronov in \cite{FGV}, which consists of a $\K$-Lie algebra $(\huaS,[\cdot,\cdot]_\huaS)$ and a $\K$-Lie algebra homomorphism $\beta:\huaS\lon \Der_\K(A)$. In this paper we only consider the case that $A$ is a commutative associative algebra.
 A Leibniz pair was also studied by Winter \cite{Win}, and called a  Lie algop.  Leibniz pairs were further studied in \cite{HL,KS}. A Lie-Rinehart algebra $({A},\L,[\cdot,\cdot]_\L,\alpha)$ naturally gives rise to a Leibniz pair by  forgetting the $A$-module structure on $\L$. We introduce the notion of an admissible representation of a Leibniz pair. If $\WRep_{\K}(\huaL)$ denotes the category of weak representations of a Lie-Rinehart algebra $\huaL$, and $\ARep_{\K}(\huaS)$ denotes the category of admissible representations of a  Leibniz pair $\huaS$, then we have the following  category equivalence:
 $$
\WRep_{\K}(\huaL)\rightleftarrows \ARep_{\K}(\huaL),
$$
where the right-hand side $\huaL$ is considered as the underlying  Leibniz pair of a Lie-Rinehart algebra.
On the other hand, a Leibniz pair also gives rise to a Lie-Rinehart algebra $\huaS\otimes_\K A$, known as the action Lie-Rinehart algebra. We show that an admissible representation of a Leibniz pair can be naturally extended to a representation of the corresponding action Lie-Rinehart algebra. Actually we have the following  category equivalence:
$$
  \ARep_{\K}(\huaS)\rightleftarrows \Rep(\huaS\otimes_\K A),
$$
where $\Rep(\huaS\otimes_\K A)$ denotes the category of representations of the Lie-Rinehart algebra $\huaS\otimes_\K A$. See Remark \ref{rmk:equi} for more details about this equivalence.

  \subsection{Crossed homomorphisms}  The concept of a crossed homomorphism of Lie algebras was introduced in \cite{Lue} in the study of nonabelian extensions of Lie algebras in 1966. A special class of crossed homomorphisms are recently  called a differential operator of weight 1 in \cite{GK, GSZ}. A flat connection $1$-form of a trivial  principle bundle is naturally a crossed homomorphism.  To the best of our knowledge this concept has not been investigated for many years. Now we have to  use it in this paper. More precisely,
  by using crossed homomorphisms, we  show that the category of weak representations (resp. admissible representations) of Lie-Rinehart algebras (resp. Leibniz pairs) is a left module category over the monoidal category of representations of  Lie algebras. In particular, we obtain  bifunctors  among  categories of certain representations:
  $$
  F_{H}:\Rep_{\K}(\g)\times\WRep_{\K}(\L) \to \WRep_{\K}(\L),\quad \huaF_{H}:\Rep_{\K}(\h)\times \ARep_{\K}(\mathcal{S}) \to \ARep_{\K}(\mathcal{S}),
  $$
  which we call the actions of monoidal categories,  generalizing  Shen-Larsson constructions of representations for Cartan type Lie algebras.  Our construction sheds light on some difficult classification problems in representation theory of Lie algebras.

  We have seen the importance of crossed homomorphisms in our above construction. To better understand crossed homomorphisms and our actions of monoidal categories, we also study  deformations and cohomologies of crossed homomorphisms.   The deformation of algebraic structures began with the seminal
work of Gerstenhaber~\cite{Ge0,Ge} for associative
algebras and followed by its extension to Lie algebras by
Nijenhuis and Richardson~\cite{Nijenhuis-Richardson-1}. A suitable deformation theory of an algebraic structure needs to
follow certain general principle: on one hand, for a given
object with the algebraic structure, there should be a
differential graded Lie algebra whose Maurer-Cartan elements
characterize deformations of this object. On the other hand, there
should be a suitable cohomology so that the infinitesimal of a
formal deformation can be identified with a cohomology class.
  We successfully construct a differential graded Lie algebra such that   crossed homomorphisms are characterized  as Maurer-Cartan elements. The cohomology groups of  crossed homomorphisms are also defined to control their linear deformations.

\subsection{Outline of the paper}

In Section \ref{sec:ch}, we recall the concept of crossed homomorphisms between Lie algebras and show that there is a one-to-one correspondence between crossed homomorphisms and certain Lie algebra homomorphisms (Theorem \ref{twitst-iso}). This fact is the key ingredient in our later construction of the left module category.

In Section \ref{sec:wr}, we introduce the new concepts:  weak representations (resp. admissible representations) of Lie-Rinehart algebras (resp. Leibniz pairs).
Using crossed homomorphisms, we show that the category of weak representations (resp. admissible representations) of Lie-Rinehart algebras (resp. Leibniz pairs) is a left module category over the monoidal category of representations of Lie algebras. In particular, the corresponding bifunctor $\huaF_H$ which we call the action of monoidal categories,
is established  to give new representations of  Lie-Rinehart algebras   (resp. Leibniz pairs). See  Theorems \ref{bifunctor} and \ref{bifunctor'}. This generalizes and unifies various existing
 constructions of representations of  many Lie algebras by using this  new  bifunctor.

In Section \ref{sec:grsl},  to show the power of our action of monoidal categories $\huaF_H$  established in Section \ref{sec:wr}, we  construct some examples of crossed homomorphisms in different situations and using our action of monoidal categories to recover some known constructions representations of  various Lie algebras (see Section \ref{subsec:w}-\ref{hh}),  and to obtain new representations of generalized Witt algebras and their Lie subalgebras (See   Corollaries \ref{cor:gw}, \ref{cor:gs}, \ref{cor:gh}).
Certainly, our action of monoidal categories will be used to other situations  to give new simple representations of suitable Lie algebras.

In Section \ref{sec:deformation},   we  characterize crossed homomorphisms as Maurer-Cartan elements in a suitable differential graded Lie algebra and  introduce  the cohomology theory of crossed homomorphisms.
We use the cohomology theory of crossed homomorphisms that we established to study linear deformations of crossed homomorphisms, to prove that the  linear  deformation ${H}_t:={H}+t d_{\rho_{H}} (-{H} x)$   is trivial  for any Nijenhuis element $x$ (Theorem \ref{thm:trivial}).

We conclude our paper in Section \ref{Sect.6} by asking three questions.

As usual, we denote by $\mathbb{Z}$, $\mathbb{Z}_+$  and
$\mathbb{C}$ the sets of  all integers,
positive integers and complex numbers. All  vector spaces are over an algebraically closed field $\mathbb K$ of characteristic $0$.
\vspace{2mm}

\section{Crossed homomorphisms between Lie algebras}\label{sec:ch}

Let $(\g,[\cdot,\cdot]_\g)$ and $(\mathcal{\h},[\cdot,\cdot]_\h)$ be Lie algebras. We will denote by $\Der(\g)$ and $\Der(\h)$ the Lie algebras of derivations on $\g$ and $\h$ respectively.
A Lie algebra homomorphism $\rho:\g\lon\Der(\h)$ will be called  an {\bf action} of $\g$ on $\h$ in the sequel.

\begin{defi}\label{crossed-homo-defi}{\rm (\cite{Lue})}
  Let $\rho:\g\lon\Der(\h)$ be an action of  $(\g,[\cdot,\cdot]_\g)$  on     $(\mathcal{\h},[\cdot,\cdot]_\h)$.  A linear map ${H}:\g\to \h$ is called a {\bf crossed homomorphism with respect to the action $\rho$} if
\begin{eqnarray}\label{crossed-homo}
{H}[x,y]_\g=\rho(x)({H} y)-\rho(y)({H} x)+[{H} x,{H} y]_\h,\quad \forall x,y\in \g.
\end{eqnarray}
\end{defi}

\begin{rmk}
A crossed homomorphism from $\g$ to $\g$ with respect to the adjoint action  is also called a {\bf  differential operator of weight $1$}. See \cite{GK, GSZ} for more details.
\end{rmk}

\begin{ex}{\rm
 Let $P$ be a trivial  $G$-principle bundle over a differential manifold $M$, where $G$ is a Lie group. Let $\omega\in\Omega^1(M,\g)$ be a connection $1$-form, where $\g$ is the Lie algebra of $G$. Then $\omega$ is flat if and only if
 $d\omega+\half[\omega,\omega]_\g=0,$
 which is equivalent to
$$
 X\omega(Y)-Y\omega(X)-\omega([X,Y])+[\omega(X),\omega(Y)]_\g=0,\quad\forall X,Y\in\frkX(M).
 $$
 Therefore, a flat connection $1$-form, i.e. $\omega\in\Omega^1(M,\g)=\Hom(\frkX(M),\g\otimes C^\infty(M))$ satisfying the above equality,  is a crossed homomorphism from the Lie algebra of vector fields $\frkX(M)$ to the Lie algebra $\g\otimes C^\infty(M)$ with respect to the action $\rho$ given by
 $$
 \rho(X)(u\otimes f)=u\otimes X(f),\quad \forall X\in\frkX(M), u\in\g, f\in C^\infty(M).  \qed
 $$
 }
\end{ex}

\begin{ex}{\rm  If the action $\rho$ of $\g$ on $\h$ is zero, then any   crossed homomorphism from $\g$ to $\h$ is nothing but  a Lie algebra homomorphism.
If $\h$ is commutative, then any   crossed homomorphism from $\g$ to $\h$ is simply a derivation from $\g$ to $\h$ with respect to the representation $(\h;\rho)$.\qed
}
\end{ex}

 \begin{defi}\label{crossed-homo-homo}
      Let ${H}$ and ${H}'$ be crossed homomorphisms from $\g$ to $\h$ with respect to the action $\rho$. A {\bf homomorphism} from ${H}'$ to ${H}$ consists of two Lie algebra homomorphisms  $\phi_\g:\g\longrightarrow\g$ and $\phi_\h:\h\longrightarrow \h$ such that
      \begin{eqnarray}
        \label{homo-1}{H}\circ \phi_\g&=&\phi_\h\circ {H}',\\
        \label{homo-2}\phi_\h(\rho(x)u)&=&\rho(\phi_\g(x))(\phi_\h(u)),\quad\forall x\in\g, u\in \h.
      \end{eqnarray}
      In particular, if $\phi_\g$ and $\phi_\h$ are  invertible,  then $(\phi_\g,\phi_\h)$ is called an  {\bf isomorphism}  from ${H}'$ to ${H}$.
    \end{defi}

The following result can be also found in \cite{Lue}.
\begin{lem}
Let ${H}$ be a crossed homomorphism from $\g$ to $\h$ with respect to the action $\rho$.
  Define $\rho_{{H}}:\g\longrightarrow\gl(\h)$ by
  \begin{equation}\label{eq:newrep}
 \rho_{{H}}(x)u:=\rho(x)u+[{H} x,u]_\h,\;\;\forall x\in \g,u\in\h.
  \end{equation}
Then  $\rho_{{H}}$ is also an action of $\g$ on $\h$, i.e. $\rho_{{H}}:\g\lon\Der(\h)$ is a Lie algebra homomorphism.
  \end{lem}


We use $\g\ltimes_{\rho_{{H}}}\h $ and $\g\ltimes_{\rho}\h$ to denote the two semidirect products of $\g$ and $\h$ with respect to the actions $\rho_{H}$ and $\rho$ respectively. More precisely, we have
\begin{eqnarray*}
  ~[(x,u),(y,v)]_{\rho_{H}}&=&[x,y]_\g+\rho_{H}(x)v-\rho_{H}(y)u+[u,v]_\h,\\
   ~[(x,u),(y,v)]_{\rho}&=&[x,y]_\g+\rho(x)v-\rho(y)u+[u,v]_\h.
\end{eqnarray*}

\begin{thm}\label{twitst-iso}
Let ${H}:\g\to \h$ be a linear map   and $\rho:\g\lon\Der(\h)$ an action of $\g$ on $\h$.

\begin{itemize}\item [{\rm(a)}] Suppose that $\rho_H$ given by \eqref{eq:newrep} is an action of $\g$ on $\h$.
  Then the linear map $\hat{{H}}:\g\ltimes_{\rho_{{H}}}\h\longrightarrow\g\ltimes_{\rho}\h$ defined by
  \begin{equation}
   \hat{{H}}(x,u):=\big(x,{H} x+u\big),\;\;\forall x\in \g,u\in\h,
  \end{equation}
 is a Lie algebra isomorphism  if and only if $H$ is a crossed homomorphism from $\g$ to $\h$ with respect to the action $\rho$.
 \item[{\rm(b)}] $H$ is a crossed homomorphism from $\g$ to $\h$ with respect to the action $\rho$ if and only if the map
$\iota_H:\g\longrightarrow\g\ltimes_{\rho}\h$ defined by
  \begin{equation}
  \iota_H(x):=\big(x,{H} x\big),\;\;\forall x\in \g
  \end{equation}
is a Lie algebra homomorphism.\end{itemize}
  \end{thm}

  \emptycomment{
\tr{I like the proposition (b) very much! In fact, there is a pair of adjoint functors $\ltimes:\g\mbox{-}\Mod\lon \Lie/\g$ with respect to (left adjoint functor) $\Omega:\Lie/\g\lon \g\mbox{-}\Mod$
                \begin{eqnarray*}
             \Hom_{\g\mbox{-}\Mod}(\Omega(\frkk),V)\cong\Der_{\g}(\frkk,V)\cong\Hom_{\Lie/\g}(\frkk,\g\ltimes V),
                \end{eqnarray*}
                where, $\kappa:\frkk \stackrel{}{\lon}\g\in \Lie/\g$ and $V\in \g\mbox{-}\Mod.$ We can deem proposition (b) is a non-abelian version of above adjoint functors. But there is a question: how to construct the left adjoint functor $\Omega$ of the non-abelian version (the module of Kahler differential forms)? There is a paper "Differentials for Lie algebras" which is written by Jochen Kuttler and Arturo Pianzola.}

\yh{So one more natural question, maybe related to Rong's question, is that what happens to the theory of modules if we consider the general nonabelian extensions. In this case $H$ is not a crossed homomorphism anymore. There will be a term $R:\wedge^2\g\lon \h$ such that $R(x,y)=\rho(x)({H} y)-\rho(y)({H} x)-{H}[x,y]_\g+[{H} x,{H} y]_\h$}

\yh{From geometric point of view,  a crossed homomorphism $H$ is a flat connection. In general, if we consider the general connection, there will be $R$ show up, which is the curvature. }

}
\begin{proof} (a). Clearly  $\hat{{H}}$ is an invertible linear map.
 For all $x,y\in\g,u,v\in\h$, we have
\begin{eqnarray*}
[\hat{{H}}(x,u),\hat{{H}}(y,v)]_{\rho}&=&[(x,{H} x+u),(y,{H} y+v)]_{\rho}\\
                                          &=&([x,y]_\g,\rho(x)({H} y+v)-\rho(y)({H} x+u)+[{H} x+u,{H} y+v]_\h)\\
 &=&\big([x,y]_\g,\rho(x)v-\rho(y)u+[{H} x,v]_\h-[{H} y,u]_\h+[u,v]_\h+[{H}x,Hy]_\g\\
 &&+\rho(x)({H} y)-\rho(y)({H} x)\big),\\
   \hat{{H}}[(x,u),(y,v)]_{\rho_{{H}}}&=&([x,y]_\g,{H}[x,y]_\g+\rho_{H}(x)v-\rho_{H}(y)u +[u,v]_\h)\\
                                           &=&([x,y]_\g,{H}[x,y]_\g+\rho(x)v-\rho(y)u+[{H} x,v]_\h-[{H} y,u]_\h+[u,v]_\h).
\end{eqnarray*}

Thus, $[\hat{{H}}(x,u),\hat{{H}}(y,v)]_{\rho}=\hat{{H}}[(x,u),(y,v)]_{\rho_{{H}}}$, if and only if
\eqref{crossed-homo} holds for $H$, which is equivalent to that $H$ is a crossed homomorphism from $\g$ to $\h$ with respect to the action $\rho$.

(b) follows from the proof of (a) by taking $u=v=0$.
\end{proof}

\begin{rmk}
  In fact, crossed homomorphisms correspond to split nonabelian extensions of Lie algebras. More precisely, we consider the following nonabelian extension of Lie algebras:
$$
0\lon \h\lon \g\oplus \h \lon \g\lon 0.
$$
A  section $s:\g\lon \g\oplus\h$ must be of the form $s(x)=(x,Hx), ~x\in\g$. Statement (b) says that $s$ is a Lie algebra homomorphism if and only if $H$ is a crossed homomorphism. Such an extension is called a split nonabelian extension. See \cite{Lue} for more details.

\end{rmk}

\section{Action of monoidal categories arising from Lie-Rinehart algebras and Leibniz pairs}\label{sec:wr}
In this section, first we introduce the notion of a weak representation of a Lie-Rinehart algebra, show that  the category of weak representations of Lie-Rinehart algebras is a left module category over the monoidal category of representations of Lie algebras by using crossed homomorphisms. Then we introduce the notion of an admissible representation of a Leibniz pair and obtain similar results. In particular, the corresponding  bifunctors are called the  actions of monoidal categories for Lie-Rinehart algebras and Leibniz pairs.

\subsection{Weak representations of Lie-Rinehart algebras}\hspace{2mm}

Let ${A}$ be a commutative associative algebra over $\K$. We denote by $\Der_{\K}({A})$ the set of $\K$-linear derivations of ${A}$, i.e.
$$
\Der_\K({A})=\{  D\in\text{End}_\K(A): D(ab)=D(a)b+aD(b),   \forall a,b\in{A}\}.
$$

\begin{defi}{\rm (\cite{Ri})}
  A {\bf Lie-Rinehart algebra} over ${A}$ is a $\K$-Lie algebra $(\L,[\cdot,\cdot]_{\L})$ together with an ${A}$-module structure on $\L$ and a map $\alpha:\L\to \Der_\K({A})$ (called the anchor) which is simultaneously a $\K$-Lie algebra and an ${A}$-module homomorphism such that
\begin{eqnarray*}
[x,ay]_\L=a[x,y]_\L+\alpha(x)(a)y,\quad \forall x,y\in \L,~a\in {A}.
\end{eqnarray*}
\end{defi}
We usually denote a Lie-Rinehart algebra over ${A}$ by $({A},\L,[\cdot,\cdot]_\L,\alpha)$ or simply by $\L$.

\begin{rmk}{  It is clear that a Lie-Rinehart algebra with $\alpha=0$ is exactly a Lie ${A}$-algebra.
}
\end{rmk}

\begin{ex} \label{example}
{\rm    $({A},\Der_\K({A}),[\cdot,\cdot]_C,\alpha=\Id)$ is a Lie-Rinehart algebra, where $ [\cdot,\cdot]_C$ is the commutator bracket.\qed}
\end{ex}

\begin{ex}{\rm
  Let ${M}$ be an ${A}$-module. Denote by $\gl_A({M})$ the set of ${A}$-module homomorphisms from ${M}$ to ${M}$. It is obvious that $(\gl_A({M}),[\cdot,\cdot]_C)$ is a Lie ${A}$-algebra.\qed
}
\end{ex}

\begin{ex}{\rm
 Let ${M}$ be an ${A}$-module. A {\bf first order differential operator} on ${M}$ is a pair $(D,\sigma)$, where $D:{M}\lon{M}$ is a $\K$-linear map and $\sigma=\sigma_{D}\in\Der_\K({A})$, satisfying the following compatibility condition:
 \begin{equation}D(am)=aD(m)+\sigma(a)m,\quad\forall a\in{A},m\in{M}.
 \end{equation}
 Denote by $\D({M})$ the set of first order differential operators on ${M}$. It is obvious that $\D({M})$ is an ${A}$-module. Define a bracket operation $[\cdot,\cdot]_C$ on $\D({M})$ by
 \begin{equation}
 [(D_1,\sigma_1),(D_2,\sigma_2)]_C:=(D_1\circ D_2-D_2\circ D_1,\sigma_1\circ \sigma_2-\sigma_2\circ \sigma_1),\quad\forall (D_1,\sigma_1),(D_2,\sigma_2)\in\D({M}),
 \end{equation}
and an ${A}$-module homomorphism $\pr:\D({M})\lon \Der_\K({A})$ by $\pr(D,\sigma)=\sigma$ for all $(D,\sigma)\in\D({M})$. Then   $({A},\D({M}),[\cdot,\cdot]_C,\alpha=\pr)$ is a Lie-Rinehart algebra.\qed
 }
\end{ex}

\begin{rmk}\label{rmk:ca}
 Let ${M}$ be an ${A}$-module. It is straightforward to see that we have a semidirect product commutative associative algebra $A\ltimes M$, where the multiplication is given by
 $$
 (a,m)\cdot (b,n)=(ab,an+bm),\quad \forall a,b\in A,~m,n\in M.
 $$
 Then  $(D,\sigma)$ is a first order differential operator on $M$ if and only if $(\sigma,D)$ is a derivation on the commutative associative algebra $A\ltimes M$. This result is the algebraic counterpart of the fact that a first order differential operator on a vector bundle $E$ can be viewed as a linear vector field on the dual bundle $E^*$. In fact, functions on the vector bundle $E^*\to N$ are generated by $C^\infty(N)$ and $\Gamma(E)$, while the latter are fibre linear functions on $E^*$. Since a first order differential operator maps a fibre linear function to a fibre linear function, so it is viewed as a linear vector field on $E^*.$
\end{rmk}

\begin{defi}
\begin{itemize}
\item[\rm(i)]
Let $({A},\L,[\cdot,\cdot]_\L,\alpha)$  and $({A},\L',[\cdot,\cdot]_{\L'},\alpha')$ be Lie-Rinehart algebras. A {\bf Lie-Rinehart weak homomorphism} is a $\K$-Lie algebra homomorphism  $f:\L \to \L'$   such that $\alpha'\circ f=\alpha$.

\item[\rm(ii)]
A Lie-Rinehart weak homomorphism $f$ is called a {\bf Lie-Rinehart homomorphism} if $f$ is also an ${A}$-module homomorphism, i.e. $f(ax)=af(x)$, for all $a\in{A}$ and $x\in \L.$
\end{itemize}
\end{defi}

Note that zero map from $\L$ to $\L'$ is not a Lie-Rinehart weak homomorphism if $\alpha\ne0$.

\begin{pro}\label{pro:composition}
Let $f_1:({A},\L_1,[\cdot,\cdot]_{\L_1},\alpha_1)\lon({A},\L_2,[\cdot,\cdot]_{\L_2},\alpha_2)$ and $f_2:({A},\L_2,[\cdot,\cdot]_{\L_2},\alpha_2)\lon({A},\L_3,[\cdot,\cdot]_{\L_3},\alpha_3)$ be two Lie-Rinehart weak homomorphisms. Then $f_2\circ f_1$ is a Lie-Rinehart weak homomorphism from $({A},\L_1,[\cdot,\cdot]_{\L_1},\alpha_1)$ to $({A},\L_3,[\cdot,\cdot]_{\L_3},\alpha_3)$.
\end{pro}
\begin{proof} This is easy to see.
\end{proof}

We denote by $\WH(\huaL,\huaL')$ the set of weak homomorphisms from the Lie-Rinehart algebra $({A},\L,[\cdot,\cdot]_\L,\alpha)$  to $({A},\L',[\cdot,\cdot]_{\L'},\alpha')$. By Proposition \ref{pro:composition}, it is easy to see that $\WH(\huaL,\huaL)$ is a monoid.

\begin{defi}
\begin{itemize}
\item[\rm(i)]
A {\bf weak representation} of a Lie-Rinehart algebra $({A},\L,[\cdot,\cdot]_\L,\alpha)$ on an ${A}$-module ${M}$ is a Lie-Rinehart weak homomorphism $\rho:\L\lon\D({M})$. We denote a weak representation by $({M};\rho)$.

\item[\rm(ii)]
A weak representation $({M};\rho)$ is called a {\bf representation} if $\rho$ is also an ${A}$-module homomorphism, i.e. $\rho:\L\lon\D({M})$ is a Lie-Rinehart homomorphism.
\end{itemize}
\end{defi}

\begin{rmk} By a  weak representation of a Lie-Rinehart algebra $({A},\L,[\cdot,\cdot]_\L,\alpha)$ on an ${A}$-module ${M}$, it means  a $\K$-Lie algebra homomorphism $\rho:\L\lon\gl_\K({M})$ such that
$$\rho(x)(au)=a\rho(x)(u)+\alpha(x)(a)u,\quad \forall x\in\L, a\in A, u\in M.
$$
That is, $(D=\rho(x),\sigma=\alpha(x))$ is a first order differential operator  on ${M}$.
\end{rmk}

\begin{rmk}
  In \cite{H1},  Huebschmann showed that
 there is a one-one correspondence between representations of a Lie-Rinehart algebra and
 representations of its universal enveloping algebra $\huaU(A, \L) :=(A\# U(\L))/J$, where $J$ is a certain ideal of the smash product $A\# U(\L)$.
More explicitly,  it is known that $U(\L)$ is a Hopf algebra  and $A$ is  a $U(\L)$-module algebra.  Then the smash product $A\# U(\L)$ (see  \cite{Mon})  is a $\K$-vector space $A\otimes U(\L)$  with elements denoted by $a\#u$ ,  and with product defined for all $a,b\in A$  and $u,v\in U(\L)$ by
 $$
 (a\#u)(b\# v)=\sum a\alpha(u_{(1)})b\# u_{(2)}v,
 $$
 where  we use the standard Sweedler notation $\Delta(u)=\sum u_{(1)}\ot u_{(2)}$ for the coproduct $\Delta$.  The algebra $A\# U(\L)$ is also called the  Massey-Peterson
 algebra in \cite{H1}.  It is not hard to see that there is a one-to-one correspondence between weak representations of a Lie-Rinehart algebra and
 representations of the smash product $A\# U(\L)$.
\end{rmk}

\begin{ex}{\rm
Let $({A},\L,[\cdot,\cdot]_\L,\alpha)$ be a Lie-Rinehart algebra. Define $\ad:\huaL\lon\frkD(\L)$ by
$$
\ad_xy=[x,y]_\L,\quad \sigma_{\ad_x}=\alpha(x),\quad \forall x,y\in\L.
$$
Then $\ad$ is a weak representation of $\L$ on $\L$. Note that $\ad$ is generally not a representation of $\L$ on itself. \qed}
\end{ex}

\begin{defi}Let $({A},\L,[\cdot,\cdot]_\L,\alpha)$ be a Lie-Rinehart algebra,  $({M};\rho)$ and $({M}';\rho')$  two
 weak representation of $\L$. An ${A}$-module homomorphism $\phi:{M}\to{M}'$ is said to be a {\bf homomorphism of weak representations} if $\phi\circ\rho(x)=\rho'(x)\circ\phi$ for all $x\in\L.$
 \end{defi}

\begin{pro}
Let $\phi:({M};\rho)\lon({M}';\rho')$ and $\phi':({M}';\rho')\lon({M''};\rho'')$ be two homomorphisms of weak representations of $\L$. Then $\phi'\circ \phi$ is a homomorphism from  $({M};\rho)$ to $({M''};\rho'')$.
\end{pro}

\begin{proof} This is easy to see.
\end{proof}

We usually denote by $M\stackrel {\phi}{\lon}M'$ a homomorphism between the weak representations $({M};\rho)$ and $({M}';\rho')$,  denote by {\bf $\WRep_{\K}(\L)$} the category of weak representations of a Lie-Rinehart algebra $({A},\L,[\cdot,\cdot]_\L,\alpha)$ and {$\Rep_{\K}(\g)$} the category of  representations of a $\K$-Lie algebra $(\g,[\cdot,\cdot]_\g)$. It is obvious that the category of representations of a Lie-Rinehart algebra  $({A},\L,[\cdot,\cdot]_\L,\alpha)$, denoted by {\bf $\Rep(\L)$}, is a full subcategory of the category {\bf $\WRep_{\K}(\L)$}. Please note    the subtle  difference of the  two categories  $\Rep_{\K}(\L)$  and $\Rep(\L)$.

\begin{defi}{\rm (\cite{CLP1})}
Let $({A},\L,[\cdot,\cdot]_\L,\alpha)$ be a Lie-Rinehart algebra and $(\G,[\cdot,\cdot]_\G)$ a Lie ${A}$-algebra. We say that $\L$ {\bf acts on} $\G$ if   a $\K$-Lie algebra homomorphism $\rho:\L\lon\Der_\K({\G})$ is given  such that
 $$\rho(ax)=a\rho(x),\quad\rho(x)(au)=a\rho(x)u+ \alpha(x)(a)u,\quad \forall a\in A, x\in \L, u\in\G.$$
 \end{defi}

Let $({A},\L,[\cdot,\cdot]_\L,\alpha)$ be a Lie-Rinehart algebra and $(\mathcal{G},[\cdot,\cdot]_\G)$ a Lie ${A}$-algebra on
which $\L$ acts via $\rho:\L\lon\Der_\K({\G})$.
 On the ${A}$-module $\L\oplus\mathcal{G}$, define a bracket operation $[\cdot,\cdot]_{\rho}$ by
$$
[(x,u),(y,v)]_{\rho}=([x,y]_\L,\rho(x)v-\rho(y)u+[u,v]_\G),\quad \forall x,y\in \L, u,v\in\mathcal{G},
$$
and define an ${A}$-module homomorphism $\tilde{\alpha}:\L\oplus\mathcal{G}\to\Der_\K({A})$ by
$$
\tilde{\alpha}(x,u)=\alpha(x),\quad  \forall x\in \L, u\in\mathcal{G}.
$$
Then  $({A},\L\oplus\G,[\cdot,\cdot]_{\rho},\tilde{\alpha})$ is a Lie-Rinehart algebra  \cite{CLP1}, which is called the {\bf semi-direct product} of $\L$ and $\G$, and denoted by $\L\ltimes_{\rho}\mathcal{G}$.

Note that the  Lie algebra $\L\ltimes_{\rho}\mathcal{G}$ acts on the Lie algebra $(\mathcal{G},[\cdot,\cdot]_\G)$ by
\begin{eqnarray}
\tilde{\rho}(x,u)v=\rho(x)v,\quad\forall x\in\mathcal{L},~u,v\in\huaG.
\end{eqnarray}
 Then using Theorem \ref{twitst-iso} (b) we can  easily verify the following result.
\begin{pro}
  With the above notations, the projection $\pr:\L\ltimes_{\rho}\mathcal{G}\lon\mathcal{G}$ is a crossed homomorphism with respect to the action $\tilde{\rho}$.
\end{pro}

\subsection{Left module categories over monoidal categories}
\begin{pro}\label{important-twist}
Let $({A},\L,[\cdot,\cdot]_\L,\alpha)$ be a Lie-Rinehart algebra and $\rho $ an action of $\L$   on a Lie ${A}$-algebra $(\G,[\cdot,\cdot]_\G)$. For a crossed homomorphism ${H}:\huaL\lon \huaG$ between $\K$-Lie algebras, we define a $\K$-linear map $\iota_{H}: \L \to \L\ltimes_{\rho}\G$ by:
$$
\iota_{H}(x)=(x,{H} x),\quad \forall x\in \L.
$$
Then $\iota_{H}$ is a Lie-Rinehart  injective weak homomorphism from $\L$ to $\L\ltimes_{\rho}\G$.
\end{pro}
\begin{proof}
By Theorem \ref{twitst-iso} (b), we know that   $\iota_{H}$ is a $\K$-Lie algebra monomorphism. Moreover, for all $x\in\L$, we have
$$
\tilde{\alpha}(\iota_{H}(x))=\tilde{\alpha}(x,{H} x)=\alpha(x),
$$
which implies that $\alpha=\tilde{\alpha}\circ\iota_{H}$. Thus, $\iota_{H}$ is a Lie-Rinehart injective weak  homomorphism.
\end{proof}

\begin{cor}\label{Pullback}
Let $({M};\rho)$ be a Lie-Rinehart weak representation of $({A},\L\ltimes_\rho\G,[\cdot,\cdot]_\rho,\tilde{\alpha})$ and ${H}$ a crossed homomorphism from   $\L$ to $\G$. Then $({M};\rho\circ\iota_{H})$ is a Lie-Rinehart weak representation of $({A},\L,[\cdot,\cdot]_\L,\alpha)$.
\end{cor}

\begin{proof}
By Propositions \ref{pro:composition} and \ref{important-twist}, we deduce that $\rho\circ\iota_{H}:\L\stackrel{\iota_{H}}\lon\L\ltimes_{\rho}\G\stackrel{\rho}\lon\D({M})$ is a Lie-Rinehart weak homomorphism.
\end{proof}

  Let $({A},\L,[\cdot,\cdot]_\L,\alpha)$ be a Lie-Rinehart algebra and $(\mg,[\cdot,\cdot]_\g)$ a $\K$-Lie algebra. Then $\G=\mg\ot_\K {A}$ is a Lie ${A}$-algebra, where  the ${A}$-module structure and the Lie bracket $[\cdot,\cdot]_\G$ is given by
  $$
  a(g\ot b)=g\ot ab,\quad [g\ot a,h\ot b]_\G=[g,h]_\mg\ot ab,\quad\forall a,b\in {A},~g,h\in \mg.
  $$
 Moreover, the Lie-Rinehart algebra $({A},\L,[\cdot,\cdot]_\L,\alpha)$ acts on the Lie ${A}$-algebra $\mg\ot_\K  {A}$ by $\alpha$ as follows:
\begin{equation}\label{eq:actional}
 \alpha(x)(g\ot a)=g\ot \alpha(x)(a),\quad \forall\  x\in \L, a\in{A}, g\in\mg.
\end{equation}
Consequently, we have the semidirect product Lie-Rinehart algebra $({A},\L\ltimes_\alpha (\mg\ot_\K{A}),[\cdot,\cdot]_\alpha,\tilde{\alpha})$.

Let $({A},\L,[\cdot,\cdot]_\L,\alpha)$ be a Lie-Rinehart algebra and $({M};\rho)$ a Lie-Rinehart weak representation of $({A},\L,[\cdot,\cdot]_\L,\alpha)$. Let $(\mg,[\cdot,\cdot]_\g)$ be a $\K$-Lie algebra   and $(V;\theta)$ a representation of $\g$. Then ${V}\ot_\K M$ has a natural ${A}$-module structure:
$$
a(v\ot m)=v\ot am,\quad \forall\   a\in{A}, v\in V, m\in{M}.
  $$
We define a $\K$-linear map $ \rho\boxplus\theta :\L\ltimes_\alpha (\mg\ot_\K {A})\lon\gl_\K({V}\ot_\K M)$ by
\begin{eqnarray*}
(\rho\boxplus\theta)(x,g\ot a)(v\ot m):=v\ot \rho(x)m+\theta(g)v\ot am
\end{eqnarray*}
for all $x\in\L,~a\in{A},~g\in\mg,~m\in{M},~v\in V$.

\begin{lem}\label{undeformed}
 With the above notations, $({V}\ot_\K M;\rho\boxplus\theta)$ is a Lie-Rinehart weak representation of the Lie-Rinehart algebra $({A},\L\ltimes_\alpha (\mg\ot_\K {A}),[\cdot,\cdot]_\alpha,\tilde{\alpha})$.
\end{lem}
\begin{proof}
Since $\rho:\L\lon\D({M})$ and $\theta:\g\lon\gl(V)$ are $\K$-Lie algebra homomorphisms. For all $a,b\in{A},x,y\in\L,g,h\in\mg, m\in {M},v\in V$, we have
\begin{eqnarray*}
   &&\Big([(\rho\boxplus\theta)(x,g\ot a),(\rho\boxplus\theta)(y,h\ot b)]_C-(\rho\boxplus\theta)([(x,g\ot a),(y,h\ot b)]_\alpha)\Big)(v\ot m)\\
   &=&(\rho\boxplus\theta)(x,g\ot a)\Big(v\ot \rho(y)m+\theta(h)v\ot bm \Big)-(\rho\boxplus\theta)(y,h\ot b)\Big(v\ot \rho(x)m+\theta(g)v\ot am\Big)\\
   &&-(\rho\boxplus\theta)\Big([x,y]_\L,h\ot \alpha(x)(b)-g\ot \alpha(y)(a)+[g,h]_\g\ot ab\Big)(v\ot m)\\
   &=& v\ot \rho(x)(\rho(y)m)+\theta(g)v\ot a(\rho(y)m)+\theta(h)v\ot \rho(x)(bm)+ \theta(g)(\theta(h)v)\ot a(bm)\\
   &&-v\ot \rho(y)(\rho(x)m)-\theta(h)v\ot b(\rho(x)m)-\theta(g)v\ot \rho(y)(am)- \theta(h)(\theta(g)v)\ot b(am)\\
   &&-v\ot\rho([x,y]_\L)m -\theta(h)v\ot\alpha(x)(b)m+\theta(g)v\ot \alpha(y)(a)m-\theta([g,h]_\g)v\ot (ab)m\\
   &=&0.
  \end{eqnarray*}
Therefore, we deduce that $\rho\boxplus\theta$ is a $\K$-Lie algebra homomorphism.

Furthermore, by $\rho(x)\in\D({M})$, we have
\begin{eqnarray*}
 (\rho\boxplus\theta)(x,g\ot b)\Big(a(v\ot m)\Big)&=&(\rho\boxplus\theta)(x,g\ot b)(v\ot am)\\
 &=&v\ot \rho(x)(am)+ \theta(g)v\ot a(bm)\\
 &=&v\ot \Big(a\rho(x)(m)+\alpha(x)(a)m\Big)+\theta(g)v \ot a(bm)\\
 &=&a\Big( (\rho\boxplus\theta)(x,g\ot b)(v\ot m)\Big)+\alpha(x)(a)(v\ot m),
\end{eqnarray*}
which implies that $(\rho\boxplus\theta)(x,g\ot b)\in\D(V\ot_\K{M})$ and $\tilde{\alpha}=\pr\circ(\rho\boxplus\theta)$.

Therefore, $\rho\boxplus\theta:\L\ltimes_\alpha (\mg\ot_\K {A})\lon\D(V\ot_\K{M})$ is a Lie-Rinehart weak homomorphism.
\end{proof}

\begin{cor}\label{iterated}
Let $({M};\rho)$ be a Lie-Rinehart representation of $({A},\L,[\cdot,\cdot]_\L,\alpha)$ and $(V;\theta)$ a representation of $\g$. Then $({V}\ot_\K M;\rho\boxplus\theta)$ is a Lie-Rinehart representation of $\L\ltimes_\alpha (\mg\ot_\K {A})$.
\end{cor}

\begin{proof}
Since $\rho$ is an ${A}$-module homomorphism, we have
\begin{eqnarray*}
&&\Big((\rho\boxplus\theta)(b(x,g\ot a))-b(\rho\boxplus\theta)(x,g\ot a)\Big)(v\ot m)\\
&=&v\ot \rho(bx)m+\theta(g)v\ot (ba)m-b\Big(v\ot \rho(x)m+\theta(g)v\ot am\Big)\\
&=&0, \quad \forall a,b\in{A},~x\in\L,~g\in\g,~m\in{M},~v\in V.
\end{eqnarray*}
Thus,   $\rho\boxplus\theta$ is also an ${A}$-module homomorphism.
\end{proof}

Before we give the main result of the paper, we recall the notions of a monoidal category and a left module category over a monoidal category.
\begin{defi}{\rm (\cite{Etingof})}
A {\bf monoidal category} is a $6$-tuple $(\huaC,\otimes, a,\mathbf{1},l,r)$ consists of the following data:
\begin{itemize}
\item A category $\huaC$;
\item A bifunctor $\otimes:\huaC\times \huaC\to \huaC$ called the monoidal product;
\item A natural isomorphism $a:\otimes\circ(\otimes\times \Id_\huaC)\lon\otimes\circ(\Id_\huaC\times \otimes)$ called the associativity isomorphism;
\item An object $\mathbf{1}\in \ob(\huaC)$ called the unit object;
\item A natural isomorphism $l:\otimes\circ(\mathbf{1}\times\Id_\huaC)\lon\Id_\huaC$ called the left unit isomorphism and a natural isomorphism $r:\otimes\circ(\Id_\huaC\times\mathbf{1})\lon\Id_\huaC$ called the right unit isomorphism.
\end{itemize}
These data satisfy the following two axioms:

$(1)$ the {\bf pentagon axiom}: the pentagon diagram
$$
\xymatrix{
&&((W\otimes X)\otimes Y)\otimes Z\ar[lld]_{a_{W\otimes X,Y,Z}}\ar[rrd]^{a_{W,X,Y}\otimes{\Id}_Z} &&\\
(W\otimes X)\otimes (Y\otimes Z)\ar[d]_{a_{W,X,Y\otimes Z}}&&&&(W\otimes (X\otimes Y))\otimes Z\ar[d]^{a_{W,X\otimes Y,Z}}\\
W\otimes( X\otimes (Y\otimes Z))&&&&W\otimes ((X\otimes Y)\otimes Z)\ar[llll]_{{\Id}_W\otimes a_{X,Y,Z}}\\
}
$$
commutes for the all $W,X,Y,Z\in\ob(\huaC)$.

\emptycomment{
$$
\xymatrix{
&&((W\otimes X)\otimes Y)\otimes Z\ar[lld]_{a_{W\otimes X,Y,Z}}\ar[rrd]^{a_{W,X,Y}\otimes{\Id}_Z} &&\\
(W\otimes X)\otimes (Y\otimes Z)\ar[d]_{a_{W,X,Y\otimes Z}}&&&&((W\otimes X)\otimes Y)\otimes Z\ar[d]_{a_{W,X\otimes Y,Z}}\ar[d]_{a_{W,X\otimes Y,Z}}\\
W(\otimes X\otimes (Y\otimes Z))&&&&W\otimes (X\otimes Y)\otimes Z)\ar[llll]_{{\Id}_W\otimes a_{X,Y,Z}}\\
}
$$
}

$(2)$ the   {\bf triangle axiom}: the triangle diagram
$$
\xymatrix{
    (X\otimes\mathbf{1}) \otimes Y\ar[rr]^{a_{X,\mathbf{1},Y}}\ar@/_/[dr]_{r_X\otimes {\Id}_Y} & & X\otimes(\mathbf{1} \otimes Y ) \ar@/^/[dl]^{{{\Id}_X}\otimes l_Y} \\
     & X\otimes Y &
    }
$$
commutes for the all $X,Y\in\ob(\huaC)$.

The monoidal category $\huaC$ is  {\bf strict} if the associativity isomorphism, left unit isomorphism and right unit isomorphism $a,l,r$ are all identities.
\end{defi}

\begin{ex}{\rm Let $\mathcal{C}$ be a category and $\mathcal{E}nd(\mathcal{C})$ the category of endofunctors (the functors from $\mathcal{C}$ into itself). Then $\mathcal{E}nd(\mathcal{C})$ is
a strict monoidal category with the composition of functors as the monoidal product and the identity functor as the unit object of this category.
}
\end{ex}

\begin{ex}{\rm  The category of   representations
$\Rep_{\K}(\g)$ of a $\K$-Lie algebra  $\g$ is a monoidal category: the monoidal product of $(V_1;\theta_1)$ and $(V_2;\theta_2)$ is defined
by
$$
(V_1;\theta_1)\otimes (V_2;\theta_2):=(V_1\otimes V_2;\theta_1\otimes {\Id}_{V_2}+{\Id}_{V_1}\otimes \theta_2),
$$
and the unit object $\mathbf{1}$ is the 1-dimensional trivial representation $(\K;0)$ of $\g$. Moreover, the associativity isomorphism $$
a_{(V_1;\theta_1),(V_2;\theta_2),(V_3;\theta_3)}:((V_1;\theta_1)\otimes (V_2;\theta_2))\otimes (V_3;\theta_3)\lon (V_1;\theta_1)\otimes ((V_2;\theta_2)\otimes (V_3;\theta_3))
$$
is defined by
\begin{eqnarray}
a_{(V_1;\theta_1),(V_2;\theta_2),(V_3;\theta_3)}\Big((v_1\otimes v_2)\otimes v_3\Big):=v_1\otimes (v_2\otimes v_3),\,\,\,\,\forall v_i\in V_i,~i=1,2,3,
\end{eqnarray}
the left unit isomorphism $l_{(V;\theta)}$ and the right unit isomorphism $r_{(V;\theta)}$ are defined by
\begin{eqnarray}
l_{(V;\theta)}(k\otimes v):=kv,\,\,\,\,r_{(V;\theta)}(v\otimes k):=kv,\,\,\,\,\forall k\in\K,~v\in V.\qquad \qed
\end{eqnarray}
}
\end{ex}

\begin{defi}\label{left-module-cate}{\rm (\cite{Etingof})}
Let $(\huaC,\otimes, a,\mathbf{1},l,r)$ be a monoidal category. A {\bf left module category} over $\huaC$ is a category $\M$ equipped with a bifunctor $\otimes^{\M}:\huaC\times \M\lon\M$, a natural isomorphism $a^{\M}:\otimes^{\M}\circ(\otimes\times {\Id}_\M)\lon\otimes^{\M}\circ({\Id}_\huaC\times \otimes^{\M})$, and a natural isomorphism $l^\M:\otimes^\M\circ(\mathbf{1}\times\Id_\M)\lon\Id_\M$ such that
the {\bf pentagon diagram}
$$
\xymatrix{
&&((X\otimes Y)\otimes Z)\otimes^\M M\ar[lld]_{a^\M_{X\otimes Y,Z,M}}\ar[rrd]^{a_{X,Y,Z}\otimes^\M{\Id}_M} &&\\
(X\otimes Y)\otimes^\M(Z\otimes^\M M)\ar[d]_{a^\M_{X,Y,Z\otimes^\M M}}&&&&(X\otimes (Y\otimes Z))\otimes^\M M\ar[d]^{a^\M_{X,Y\otimes Z,M}}\\
X\otimes^\M( Y\otimes^\M (Z\otimes^\M M))&&&&X\otimes^\M ((Y\otimes Z)\otimes^\M M)\ar[llll]_{{\Id}_X\otimes^\M a^\M_{Y,Z,M}}\\
}
$$
and the {\bf triangle diagram}
$$
\xymatrix{
    (X\otimes\mathbf{1}) \otimes^\M M\ar[rr]^{a^\M_{X,\mathbf{1},M}}\ar@/_/[dr]_{r_X\otimes^\M {\Id}_M} & & X\otimes^\M(\mathbf{1} \otimes^\M M ) \ar@/^/[dl]^{{{\Id}_X}\otimes^\M l^\M_M} \\
     & X\otimes^\M M &
    }
$$
commute for the all $X,Y,Z\in\ob(\huaC),~M\in\ob(\M)$.
\end{defi}

\begin{ex}{\rm
Any monoidal category $(\huaC,\otimes, a,\mathbf{1},l,r)$ is a left module category over itself. More precisely, we set $\otimes^\huaC=\otimes,~a^\huaC=a,~l^\huaC=l$. This left module category can be considered as a categorification of the regular representation of an associative algebra.\qed
}
\end{ex}

Let $({A},\L,[\cdot,\cdot]_\L,\alpha)$ be a Lie-Rinehart algebra and $\mg$ a $\K$-Lie algebra. Let ${H}$ be a crossed homomorphism from the $\K$-Lie algebra $\L$ to $\g\ot_\K{A}$ with respect to the action $\alpha$ given by \eqref{eq:actional}. For all $x\in\L$, we set $Hx=\sum_i x_i^\g\otimes x^A_i$ or $Hx=x_i^\g\otimes x^A_i$ for simplicity.

By Corollary \ref{Pullback} and Lemma \ref{undeformed}, our main theorem can be stated as follows:

\begin{thm}\label{bifunctor} Let $({A},\L,[\cdot,\cdot]_\L,\alpha)$ be a Lie-Rinehart algebra,  $(\mg,[\cdot,\cdot]_\g)$  a $\K$-Lie algebra.  Then any crossed homomorphism $H:\L\lon\g\ot_\K{A}$ induces a left module category structure of the category of weak representations   {$\WRep_{\K}(\L)$} over the monoidal category {$\Rep_{\K}(\g)$}. More precisely, the left module structure is given by
\begin{itemize}
  \item the bifunctor $F_{H}:\Rep_{\K}(\g)\times\WRep_{\K}(\L) \to \WRep_{\K}(\L)$, which is defined  on the set of objects and on the set of morphisms respectively by
\begin{eqnarray}
F_{H}\Big((V;\theta),({M};\rho)\Big)&=&({V}\ot_\K M;(\rho\boxplus\theta)\circ\iota_{H}),\\
F_{H}(V\stackrel{\psi}{\lon}V',M\stackrel{\phi}{\lon}M')&=&V\otimes M\stackrel{\psi\otimes \phi}{\lon}V'\otimes M',
\end{eqnarray}
 for $(V;\theta),(V';\theta')\in \Rep_{\K}(\g),~({M};\rho),({M}';\rho')\in\WRep_{\K}(\L),$ representation homomorphism $V\stackrel{\psi}{\lon}V'$ of the $\K$-Lie algebra $(\mg,[\cdot,\cdot]_\g)$ and weak representation homomorphism $M\stackrel{\phi}{\lon}M'$ of the Lie-Rinehart algebra $({A},\L,[\cdot,\cdot]_\L,\alpha)$;

\item the natural isomorphism $$a_{(V_1;\theta_1),(V_2;\theta_2),(M;\rho)}:F_H\big((V_1;\theta_1)\otimes(V_2;\theta_2), (M;\rho)\big)\lon F_H\Big((V_1;\theta_1),F_H\big((V_2;\theta_2), (M;\rho)\big)\Big),$$ which is defined by
\begin{eqnarray}
a_{(V_1;\theta_1),(V_2;\theta_2),(M;\rho)}((v_1\otimes v_2)\otimes m)=v_1\otimes (v_2\otimes m),
\end{eqnarray}
\item  the natural isomorphism $l_{({M};\rho)}:F_H\big((\K;0),({M};\rho))\lon ({M};\rho)$, which is defined by
\begin{eqnarray}
l_{({M};\rho)}(k\otimes m)=km.
\end{eqnarray}
\end{itemize}
\end{thm}

\begin{proof}

By Corollary \ref{Pullback} and Lemma \ref{undeformed}, $({V}\ot_\K M;(\rho\boxplus\theta)\circ\iota_{H})$ is a weak representation of $\huaL$. Thus, $F_H$ is well-defined on the set of objects.
To see that $F_H$ is also well-defined on the set of morphisms, we need to show that the linear map $\psi\otimes \phi:V\otimes M\lon V'\otimes M'$ is indeed a homomorphism from $({V}\ot_\K M;(\rho\boxplus\theta)\circ\iota_{H})$ to $({V}'\ot_\K M';(\rho'\boxplus\theta')\circ\iota_{H})$. In fact, for all $a\in A,~v\in V,~m\in M$, we have
\begin{eqnarray*}
(\psi\otimes \phi)\Big(a(v\otimes m)\Big)-a\Big((\psi\otimes \phi)(v\otimes m)\Big)
&=&(\psi\otimes \phi)(v\otimes am)-a\Big(\psi(v)\otimes \phi(m)\Big)\\
&=&\psi(v)\otimes \phi(am)\otimes -\psi(v)\otimes a\phi(m)\\
&=&0.
\end{eqnarray*}
For all $x\in\g,~v\in V$ and $m\in M$, we have
\begin{eqnarray*}
&&(\psi\otimes \phi)\Big(\big((\rho\boxplus\theta)\iota_{H}(x)\big)(v\otimes m)\Big)-\Big((\rho'\boxplus\theta')\iota_{H}(x)\Big)\Big((\psi\otimes \phi)(v\otimes m)\Big)\\
&=&(\psi\otimes \phi)\Big(\big((\rho\boxplus\theta)(x,x_i^\g\otimes x_i^A)\big)(v\otimes m)\Big)-\Big((\rho'\boxplus\theta')(x,x_i^\g\otimes x_i^A)\Big)\Big(\psi(v)\otimes \phi(m)\Big)\\
&=&(\psi\otimes \phi)\Big(v\otimes\rho(x)m+\theta(x_i^\g)v\otimes x_i^Am\Big)-\Big(\psi(v)\otimes \rho'(x)\phi(m)+\theta'(x_i^\g)\psi(v)\otimes x_i^A\phi(m)\Big)\\
&=&\Big(\psi(v)\otimes\phi(\rho(x)m)+\psi(\theta(x_i^\g)v)\otimes  \phi(x_i^Am)\Big)-\Big(\psi(v)\otimes \rho'(x)\phi(m)+\theta'(x_i^\g)\psi(v)\otimes x_i^A\phi(m)\Big)\\
&=&0.
\end{eqnarray*}
Thus, we obtain that $F_{H}(\psi,\phi)=\psi\otimes \phi$ is a homomorphism of the weak representations. Moreover, by straightforward computations, we deduce that $F_{H}$ preserves identity morphisms and composite morphisms. Therefore, $F_{H}$ is a bifunctor.

Let $(V_1;\theta_1)$ and $(V_2;\theta_2)$ be representations of the $\K$-Lie algebra $\g$ and $({M};\rho)$ be a weak representation of the Lie-Rinehart algebra $({A},\L,[\cdot,\cdot]_\L,\alpha)$. For all $b\in A,~v_1\in V_1,~v_2\in V_2$ and $m\in M$, we have
\begin{eqnarray*}
&&a_{(V_1;\theta_1),(V_2;\theta_2),(M;\rho)}\Big(b\big((v_1\otimes v_2)\otimes m\big)\Big)-ba_{(V_1;\theta_1),(V_2;\theta_2),(M;\rho)}\Big((v_1\otimes v_2)\otimes m\Big)\\
&=&a_{(V_1;\theta_1),(V_2;\theta_2),(M;\rho)}\Big((v_1\otimes v_2)\otimes bm\Big)-b\Big(v_1\otimes (v_2\otimes m)\Big)\\
&=&v_1\otimes (v_2\otimes bm)-v_1\otimes b(v_2\otimes m)\\
&=&0.
\end{eqnarray*}
For all $x\in\L,~v_1\in V_1,~v_2\in V_2$ and $m\in M$, we have
\begin{eqnarray*}
&&a_{(V_1;\theta_1),(V_2;\theta_2),(M;\rho)}\Big(\Big(\Big(\rho\boxplus\Big(\theta_1\otimes {\Id}_{V_2}+{\Id}_{V_1}\otimes \theta_2\Big)\Big)\iota_{H}(x)\Big)\big((v_1\otimes v_2)\otimes m\big)\Big)\\
&&-\Big(\Big(\Big((\rho\boxplus\theta_2)\circ\iota_{H}\Big)\boxplus\theta_1\Big)\iota_{H}(x)\Big)a_{(V_1;\theta_1),(V_2;\theta_2),(M;\rho)}((v_1\otimes v_2)\otimes m)\\
&=&a_{(V_1;\theta_1),(V_2;\theta_2),(M;\rho)}\Big(\Big(\Big(\rho\boxplus\Big(\theta_1\otimes {\Id}_{V_2}+{\Id}_{V_1}\otimes \theta_2\Big)\Big)(x,x_i^\g\otimes x_i^A)\Big)((v_1\otimes v_2)\otimes m)\Big)\\
&&-\Big(\Big(\Big((\rho\boxplus\theta_2)\circ\iota_{H}\Big)\boxplus\theta_1\Big)\iota_{H}(x)\Big)(v_1\otimes (v_2\otimes m))\\
&=&a_{(V_1;\theta_1),(V_2;\theta_2),(M;\rho)}\Big((v_1\otimes v_2)\otimes\rho(x)m+(\theta_1\otimes {\Id}_{V_2}+{\Id}_{V_1}\otimes \theta_2)(x_i^\g)(v_1\otimes v_2)\otimes x_i^Am\Big)\\&&-\Big(\Big(\Big((\rho\boxplus\theta_2)\circ\iota_{H}\Big)\boxplus\theta_1\Big)(x,x_i^\g\otimes x_i^A)\Big)(v_1\otimes (v_2\otimes m))\\
&=&v_1\otimes (v_2\otimes\rho(x)m)+\theta_1(x_i^\g)v_1\otimes (v_2\otimes x_i^Am)+v_1\otimes (\theta_2(x_i^\g)v_2\otimes x_i^Am)\\
&&-\Big(v_1\otimes\big((\rho\boxplus\theta_2)\iota_{H}(x)(v_2\otimes m)\big)+\theta_1(x_i^\g)v_1\otimes x_i^A(v_2\otimes m)\Big)\\
&=&v_1\otimes (v_2\otimes\rho(x)m)+\theta_1(x_i^\g)v_1\otimes (v_2\otimes x_i^Am)+v_1\otimes (\theta_2(x_i^\g)v_2\otimes x_i^Am)\\
&&-\Big(v_1\otimes\big(v_2\otimes\rho(x)m+\theta_2(x_i^\g)v_2\otimes x_i^Am\big)+\theta_1(x_i^\g)v_1\otimes (v_2\otimes x_i^Am)\Big)\\
&=&0.
\end{eqnarray*}

Thus, we obtain that $a_{(V_1;\theta_1),(V_2;\theta_2),(M;\rho)}$ is a homomorphism of the weak representations. Moreover, by straightforward computations, we obtain that $a_{(V_1;\theta_1),(V_2;\theta_2),(M;\rho)}$ is a natural isomorphism and satisfies the pentagon diagram in Definition \ref{left-module-cate}.

Let $({M};\rho)$ be a weak representation of the Lie-Rinehart algebra $({A},\L,[\cdot,\cdot]_\L,\alpha)$.  We have
\begin{eqnarray*}
l_{({M};\rho)}(a(k\otimes m))=l_{({M};\rho)}(k\otimes am)=k(am)=a(km)=al_{({M};\rho)}(k\otimes m), \quad \forall a\in A,~k\in\K, ~m\in M.
\end{eqnarray*}
For all $x\in\L,~k\in\K$ and $m\in M$, we have
\begin{eqnarray*}
&&l_{({M};\rho)}\Big(\big((\rho\boxplus 0)\iota_{H}(x)\big)(k\otimes m)\Big)-\rho(x)\Big(l_{({M};\rho)}(k\otimes m)\Big)\\
&=&l_{({M};\rho)}\Big(\big((\rho\boxplus 0)(x,x_i^\g\otimes x_i^A)\big)(k\otimes m)\Big)-\rho(x)(km)
=l_{({M};\rho)}\Big(k\otimes\rho(x)m\Big)-\rho(x)(km)\\
&=&k(\rho(x)m)-\rho(x)(km)
=0.
\end{eqnarray*}
Thus, we deduce that $l_{({M};\rho)}$ is a homomorphism of   weak representations. Moreover, by straightforward computations, we obtain that $l_{({M};\rho)}$ is a natural isomorphism and satisfies the triangle diagram in Definition \ref{left-module-cate}. The proof is finished.
\end{proof}

Since $({A};\alpha)$ is  a representation of a Lie-Rinehart algebra $({A},\L,[\cdot,\cdot]_\L,\alpha)$, which is known as the {\bf natural representation}, we obtain the following result.

\begin{cor}\label{Shen-Larsson -functors}
Let $({A},\L,[\cdot,\cdot]_\L,\alpha)$ be a Lie-Rinehart algebra,  $(\mg,[\cdot,\cdot]_\g)$  a $\K$-Lie algebra and ${H}$  a crossed homomorphism from   $\L$ to $\g\ot_\K{A}$. Then we have   a functor     $$\aligned &F_{{H}}^A:\Rep_{\K}(\g)\to \WRep_{\K}(\L),\\
& (V;\theta)\mapsto (V\ot_\K{A}; ({\alpha}\boxplus \theta)\circ\iota_{H}),\quad\forall(V;\theta)\in \Rep_{\K}(\g).\endaligned
$$
\end{cor}

We can also have a very useful functor on $\WRep_{\K}(\L)$ as follows.

\begin{cor}\label{twisting-functors}
Let $({A},\L,[\cdot,\cdot]_\L,\alpha)$ be a Lie-Rinehart algebra, $(\mg,[\cdot,\cdot]_\g)$ a $\K$-Lie algebra, ${H}$  a crossed homomorphism from  $\L$ to $\g\ot_\K{A}$, and $(V;\theta)$  a given representation of $\g$. Then we have   a functor
$$\aligned &F_{H}^\theta :\WRep_{\K}(\L)\to \WRep_{\K}(\L),\\
& ({M};\rho)\mapsto (V\ot_\K{M} ;({\rho}\boxplus \theta)\circ\iota_{H}),\quad\forall({M};\rho)\in\WRep_{\K}(\L).\endaligned
$$
\end{cor}

A special but very interesting case of the above result is that $(V;\theta)=(\g;\ad)$. In the next section we will show that Corollary \ref{twisting-functors} is a very efficient way to construct interesting modules from easy modules.

\subsection{Admissible representations of Leibniz pairs}

In this subsection, we introduce the notion of an admissible representation  of a Leibniz pair.  In the sequel, $A$ is always a commutative associative algebra.  The notion of a Leibniz pair was originally given in \cite{FGV}.

\begin{defi}{\rm(\cite{FGV})}
  A {\bf Leibniz pair}  consists of a $\K$-Lie algebra $(\huaS,[\cdot,\cdot]_\huaS)$ and a $\K$-Lie algebra homomorphism $\beta:\huaS\lon \Der_\K(A)$.
\end{defi}
We denote  a Leibniz pair by  $(A,\huaS,[\cdot,\cdot]_\huaS,\beta)$, or simply by $\huaS$.

 \begin{defi}
An {\bf admissible representation} of a Leibniz pair $(A,\huaS,[\cdot,\cdot]_\huaS,\beta)$ consists of an $A$-module $M$  and a $\K$-Lie algebra homomorphism $\rho:\mathcal{S}\to \gl_\K(M)$   such that
\begin{equation}\label{eq:conanchor}\rho(x)(am)=a\rho(x)m+\beta(x)(a)m,\quad \forall x\in \mathcal{S}, a\in A, m\in M.\end{equation}
 \end{defi}

\begin{defi}Let $(A,\huaS,[\cdot,\cdot]_\huaS,\beta)$ be a Leibniz pair,  $({M};\rho)$ and $({M}';\rho')$  two
 admissible representation of $\huaS$. An ${A}$-module homomorphism $\phi:{M}\to{M}'$ is said to be a {\bf homomorphism of admissible representations} if $\phi\circ\rho(x)=\rho'(x)\circ\phi$ for all $x\in \huaS.$
 \end{defi}

 Admissible representations of Leibniz pairs are like weak representations of  Lie-Rinehart algebras. We use $\ARep_{\K}(\mathcal{S})$ to denote the category of admissible representations of  $\mathcal{S}$.

 It is straightforward to obtain the following result.
 \begin{lem}
Let $(A,\huaS,[\cdot,\cdot]_\huaS,\beta)$ be a Leibniz pair, $M$ an $A$-module and $\rho:\mathcal{S}\to \gl_\K(M)$ a $\K$-linear map. Then $(M;\rho)$ is an  admissible representation of $\huaS$ if and only if $(A\ltimes M, \huaS\oplus M,[\cdot,\cdot]_\rho,\hat{\beta})$ is a Leibniz pair, where $A\ltimes M$ is the commutative associative algebra given in Remark \ref{rmk:ca}, $[\cdot,\cdot]_\rho$ is the semidirect product Lie bracket and $\hat{\beta}:\huaS\oplus M\lon\Der_\K(A\ltimes M)$ is defined by
$$
\hat{\beta}(x,m)(a,n):=(\beta(x)a,\rho(x)n),\quad\forall x\in \huaS, a\in A, m,n\in M.
$$
 \end{lem}

It is obvious that any Lie-Rinehart algebra is a  Leibniz pair. A weak representation of a Lie-Rinehart algebra is naturally an admissible representation of the underlying Leibniz pair. Actually we have the following  category equivalence:
 $$
\WRep_{\K}(\huaL)\rightleftarrows \ARep_{\K}(\huaL),
$$
where the right-hand side $\huaL$ is considered as a Leibniz pair.

 Conversely, given a Leibniz pair $(A,\huaS,[\cdot,\cdot]_\huaS,\beta)$, we also have
  an action Lie-Rinehart algebra $(A,\huaS\otimes_\K A,[\cdot,\cdot],\alpha)$,  where  the ${A}$-module structure and the $\K$-Lie bracket $[\cdot,\cdot]$ are given by
  $$
  a(x\ot b)=x\ot ab,\quad [x\ot a,y\ot b]=[x,y]_\huaS\ot ab+y\otimes (a\beta(x)b)-x\otimes (b\beta(y)a),
  $$
and an $A$-module homomorphism $\alpha:\huaS\otimes_\K A\lon\Der_\K({A})$ is defined by
$
\alpha(x\ot a):=a\beta(x)
$
for all  $a,b\in {A},~x,y\in \huaS.$ Furthermore, we obtain the following result.
\begin{pro}\label{pro:constructionbarrho}
  Let $(M;\rho)$ be an   admissible representation   of a Leibniz pair $(A,\huaS,[\cdot,\cdot]_\huaS,\beta)$. Define $\overline{\rho}:\huaS\otimes_\K A\lon \gl_\K(M)$ by
  $$
  \overline{\rho}(x\ot a):=a\rho(x),\quad \forall x\in\huaS, a\in A.
  $$
  Then $(M;\overline{\rho})$ is a representation of the Lie-Rinehart algebra $(A,\huaS\otimes_\K A,[\cdot,\cdot],\alpha)$.
\end{pro}
 \begin{proof}
  First it is obvious that $ \overline{\rho}$ is an $A$-module homomorphism from $\huaS\otimes_\K A$ to $\gl_\K(M)$. Then it is straightforward to deduce that $\overline{\rho}$ is a $\K$-Lie algebra homomorphism. Finally, by \eqref{eq:conanchor}, we deduce that
  $$
  \overline{\rho}(x\ot a)(bm)=a\rho(x)(bm)=a\big(b\rho(x)m+\beta(x)(b)m\big)=b\overline{\rho}(x\ot a)m+\alpha(x\ot a)(b)m.
  $$
Thus,  $(M;\overline{\rho})$ is a representation of the Lie-Rinehart algebra $\huaS\otimes_\K A$.
 \end{proof}

\begin{rmk}\label{rmk:equi}
Actually we have the following  category equivalence if  $A$ is unital:
	$$
	\ARep_{\K}(\huaS)\rightleftarrows \Rep(\huaS\otimes_\K A).
	$$
First the construction of Proposition \ref{pro:constructionbarrho} can be easily enhanced to a functor. In fact, assume that $\phi:{M}\to{M}'$  is a   homomorphism of admissible representations  of a Leibniz pair $(A,\huaS,[\cdot,\cdot]_\huaS,\beta)$, then  it is straightforward to deduce that  $$\phi\circ\overline{\rho}(x\otimes a)=\overline{\rho'}(x\otimes a)\circ\phi$$ for all $x\in \huaS, a\in A.$ Thus,  $\phi:{M}\to{M}'$  is also a   homomorphism of  representations of the Lie-Rinehart algebra $\huaS\otimes_\K A.$ So we obtain a functor $P:\ARep_{\K}(\huaS)\to \Rep(\huaS\otimes_\K A)$, which is defined on the sets of objects and morphisms respectively by
\begin{eqnarray*}
  P(M;\rho)&=&(M;\overline{\rho}),\\
  P(\phi:{M}\to{M}')&=&(\phi:{M}\to{M}').
\end{eqnarray*}
Conversely, let $(M;\rho)$ be a representation of the Lie-Rinehart algebra $\huaS\otimes_\K A.$ Define $\widetilde{ \rho}:\huaS \lon \gl_\K(M)$ by
  $$
   \widetilde{\rho}(x ):= \rho(x\otimes 1),\quad \forall x\in\huaS.
  $$
  Similar as the above discussion, this can also be enhanced to a functor and give the equivalence between the categories $\ARep_{\K}(\huaS)$ and $ \Rep(\huaS\otimes_\K A)$.
\end{rmk}
	
Let $(A,\huaS,[\cdot,\cdot]_\huaS,\beta)$ be a Leibniz pair and $\frkh$ be a $\K$-Lie algebra. Then $(A,\mathcal{S}\oplus  (\h\ot_\K {A}),[\cdot,\cdot],\tilde{\beta})$ is a Leibniz pair, where the $\K$-Lie algebra structure on $\mathcal{S}\oplus  (\h\ot_\K {A})$ is given by
$$
[(x,g\otimes a),(y,h\otimes b)]=([x,y]_\huaS,h\otimes\beta(x)(b)-g\otimes\beta(y)(a)+[g,h]_\h\otimes ab),\quad \forall x,y\in \huaS, g\otimes a, h\otimes b\in\h\otimes_\K A,
$$
and $\tilde{\beta}:\huaS\oplus (\h\ot_\K {A})\lon \Der_\K(A)$ is given by
  $$
  \tilde{\beta}(x,g\otimes a)=\beta(x).
  $$
Denote this Leibniz pair by $\mathcal{S}\ltimes_\beta (\h\ot_\K {A})$.

  Let   $({M};\rho)$ be an admissible representation over $\mathcal{S}$  and $(V;\theta)$ a representation of a $\K$-Lie algebra $\h$. Then ${V}\ot_\K M$ has a natural ${A}$-module structure:
$$
a(v\ot m)=v\ot am,\quad \forall\   a\in{A}, v\in V, m\in{M}.
  $$
We define a $\K$-linear map $ \rho\boxplus\theta :\mathcal{S}\ltimes_\beta (\h\ot_\K {A})\lon\gl_\K({V}\ot_\K M)$ by
\begin{eqnarray*}
(\rho\boxplus\theta)(x,g\ot a)(v\ot m):=v\ot \rho(x)m+\theta(g)v\ot am, \quad \forall x\in\huaS,~a\in{A},~g\in\h,~m\in{M},~v\in V.
\end{eqnarray*}

Then it is straightforward to verify the following result.

\begin{lem}\label{undeformed1}
 With the above notations,
  $({V}\ot_\K M;\rho\boxplus\theta)$ is an admissible representation of the Leibniz pair $\mathcal{S}\ltimes_\beta (\h\ot_\K {A})$.
\end{lem}

Let ${H}$ be a crossed homomorphism from  the $\K$-Lie algebra $\mathcal{S}$ to $ \h\ot_\K{A}$.
Then we have the Lie algebra homomorphism
$$\aligned  &\iota_{H}: \mathcal{S} \to \mathcal{S}\ltimes_\beta (\h\ot_\K {A})\\
&\iota_{H}(x)=(x,{H} x),\quad \forall x\in \mathcal{S}.\endaligned
$$

Similar to  Theorem \ref{bifunctor}, we have the following result.

 \begin{thm}\label{bifunctor'}  Any crossed homomorphism $H:\huaS\lon\h\ot_\K{A}$ induces a left module category structure of the category of admissible representations   {$\ARep_{\K}(\huaS)$} over the monoidal category {$\Rep_{\K}(\h)$}
 $$\aligned &\huaF_{H}:\Rep_{\K}(\h)\times \ARep_{\K}(\mathcal{S}) \to \ARep_{\K}(\mathcal{S})\\
&\huaF_{H}\Big((V;\theta),({M};\rho)\Big)=({V}\ot_\K M;(\rho\boxplus\theta)\circ\iota_{H}).
\endaligned$$\end{thm}

 \begin{proof} We verify that the representation $({V}\ot_\K M;(\rho\boxplus\theta)\circ\iota_{H})$ satisfies \eqref{eq:conanchor}.
 For any $x\in \mathcal{S}, a\in A, v\in V, m\in M$. Suppose $H(x)=Hx=\sum_i x_i^\h\otimes x^A_i$ or $Hx=x_i^\h\otimes x^A_i$ for simplicity.
 Then
 $$\aligned ((\rho\boxplus\theta)\circ\iota_{H})(x)(a(v\otimes m))
 =&((\rho\boxplus\theta) (x, x_i^\h\otimes x^A_i)(v\otimes am)\\
 =&v\otimes \rho(x)(am)+ \theta(x_i^\h) v\otimes x^A_iam\\
 =&a\Big(v\otimes \rho(x)(m)+ \theta(x_i^\h) v\otimes x^A_im\Big)+\beta(x)(a)(v\otimes m)\\
 =&a((\rho\boxplus\theta)\circ\iota_{H})(x)(v\otimes m)+\beta(x)(a)(v\otimes m).\endaligned
 $$
The proof is similar to Theorem \ref{bifunctor}. So the details will be omitted.
 \end{proof}
Since $({A};\beta)$ is  an admissible  representation of a Leibniz pair  $(A,\huaS,[\cdot,\cdot]_\huaS,\beta)$, we obtain the following result.

\begin{cor}\label{Shen-Larsson-functors'}
Let $(A,\huaS,[\cdot,\cdot]_\huaS,\beta)$ be a Leibniz pair,  $(\h,[\cdot,\cdot]_\h)$  a $\K$-Lie algebra and ${H}$  a crossed homomorphism from   $\huaS$ to $\h\ot_\K{A}$. Then we have   a functor     $$\aligned &\huaF_{{H}}^A:\Rep_{\K}(\h)\to \ARep_{\K}(\huaS),\\
& (V;\theta)\mapsto (V\ot_\K{A}; ({\beta}\boxplus \theta)\circ\iota_{H}),\quad\forall(V;\theta)\in \Rep_{\K}(\h).\endaligned
$$
\end{cor}

We can also have a very useful functor on $\WRep_{\K}(\L)$ as follows.

\begin{cor}\label{twisting-functors'}
Let $(A,\huaS,[\cdot,\cdot]_\huaS,\beta)$ be a Leibniz pair,  $(\h,[\cdot,\cdot]_\h)$  a $\K$-Lie algebra, ${H}$  a crossed homomorphism from  $\huaS$ to $\h\ot_\K{A}$, and $(V;\theta)$  a given representation of $\h$. Then we have   a functor
$$\aligned &\huaF_{H}^\theta :\ARep_{\K}(\huaS)\to \ARep_{\K}(\huaS),\\
& ({M};\rho)\mapsto (V\ot_\K{M} ;({\rho}\boxplus \theta)\circ\iota_{H}),\quad\forall({M};\rho)\in\ARep_{\K}(\huaS).\endaligned
$$
\end{cor}

A special but very interesting case of the above result is that $(V;\theta)=(\h;\ad)$.

According to Corollaries \ref{Shen-Larsson -functors} and \ref{Shen-Larsson-functors'},   the bifunctors $F_H$  in Theorems \ref{bifunctor} and $\huaF_H$ given in Theorems   \ref{bifunctor'}
are the actions of monoidal categories.

\section{Representations of Cartan type Lie algebras}\label{sec:grsl}

From the definition of a crossed homomorphism we see that it is generally   hard to find nontrivial crossed homomorphisms. Next we will show you some examples of crossed homomorphisms and their tremendous power  in obtaining new irreducible modules via results in the previous section.

\subsection{Shen-Larsson functors of  Witt type}\label{subsec:w}

For $n\geq1$, recall the Witt algebra $\huaW_n=\Der (A_n)$ over the
Laurent polynomial algebra $A_n=\C[x_1^{\pm1}, \cdots,x_n^{\pm1}]$,  which can be interpreted as the Lie algebra of (complex-valued)
polynomial vector fields on an $n$-dimensional torus.
Let $\partial_i=\frac{\partial}{\partial x_i}$ be the partial derivation with respect to the variable $x_i$ for $i=1,2,\dots,n$, denote $d_i=x_i\partial_i$, and
$x^r=x_1^{r_1}x_2^{r_2}\cdots x_n^{r_n}$ for $r=(r_1,r_2,\cdots, r_n)^T\in\mathbb{Z}^n$. Then
$$\huaW_n= {\text{span}}\{x^rd_i\mid r\in\mathbb{Z}^n,
1\leq i\leq n\}$$ with the Lie bracket:
$$
[x^rd_i,x^sd_j]_{\huaW_n}=s_ix^{r+s}d_j-r_jx^{r+s}d_i,\quad \forall\,1\leq i,j\leq n, r,s\in \mathbb {Z}^n.
$$
Obviously,  $(A_n,\huaW_n, [\cdot,\cdot]_{\huaW_n},\Id)$ is a Lie-Rinehart algebra. Certainly $ (A_n; \Id)$ is the natural representation of the Lie-Rinehart algebra $(A_n,\huaW_n, [\cdot,\cdot]_{\huaW_n},\Id)$. Let $\g=\gl_n$ be the Lie algebra of all $n \times n$ complex matrices. Then $\huaG=\gl_n\otimes A_n$ is a Lie $A_n$-algebra.  For $1\leq i, j \leq n$, we use $E_{ij}$ to denote the $n\times n$ matrix with $1$
at the $(i, j)$ entry and zeros elsewhere.

\begin{lem}\label{CH-W}
 The linear map $H:\huaW_n\to  \gl_n\ot A_n$  defined by
$$
H(x^rd_j)= \sum_{i=1}^n r_i E_{ij}\ot x^r,\quad\forall r\in \Z^n, 1\leq j\leq n
$$
 is a crossed homomorphism from $\huaW_n$ to $\gl_n\otimes A_n$.
\end{lem}

\begin{proof}This follows from (2.5) in \cite{GLLZ} (or (2.3) and Lemma 2.1 in \cite{LLZ}), and Theorem \ref{twitst-iso}.
\end{proof}

 By Lemma \ref{CH-W} and Corollary \ref{Shen-Larsson -functors}, we obtain the following result.

 \begin{cor}\label{cor:trivialm} We have a functor $F_{{H}}^{A_n}:\Rep_{\C}(\gl_n)\to \WRep_{\C}(\huaW_n)$  given by
$$
F_{H}^{A_n}(V;\theta)=(V\ot_\C{A_n}; ({\Id}\boxplus \theta)\circ\iota_{H}),\quad\forall(V;\theta)\in \Rep_{\C}(\gl_n).
$$
\end{cor}

\begin{rmk}{ By forgetting the $A_n$-module structure, the corresponding functor $F_H^{A_n}$ is the well-known  Shen-Larsson   functor of type $(\huaW_n,\gl_n)$, introduced  by
Shen \cite{Sh}  (over polynomial algebras), and   Larsson   \cite{L3} (over  Laurent polynomial algebras), independently in   different settings. For any simple $\gl_n$-module $V$ the simplicity of the $\huaW_n$-module
$F_{H}^{A_n}(V;\theta)$ was determined in \cite{E1,GZ,LZ0}. In particular, simple $\huaW_n$-modules of this class (with $V$ to be simple finite-dimensional $\gl_n$-modules) are all the simple Harish-Chandra $\huaW_n$-modules \cite{BF0}.
}
\end{rmk}

Let ${\huaA}_n=\C[x_1^{\pm1},\cdots,x_n^{\pm1},\partial_1,\cdots,\partial_n]$ be the Weyl algebra, which is  the universal enveloping algebra of  the Lie-Rinehart algebra $(A_n,\huaW_n, [\cdot,\cdot]_{\huaW_n},\Id)$. Let $(P;\rho)$ be a  representation of $\huaA_n$. It is obvious that  $(P;\rho|_{\huaW_n})$ is a $\huaW_n$-module. By   Lemma \ref{CH-W} and Corollary \ref{twisting-functors}, we obtain the following result.

\begin{cor}
We have a functor  $F_{{H}}^P:\Rep_{\C}(\gl_n)\to \WRep_{\C}(\huaW_n)$  given by
$$
F_{H}^P(V;\theta)=(V\ot_\C{P}; ({\rho|_{\huaW_n}}\boxplus\theta )\circ\iota_{H}),\quad\forall(V;\theta)\in \Rep_{\C}(\gl_n).
$$
\end{cor}

\begin{rmk}{  The functor $F_{H}^P$,  introduced by Liu, Lu and Zhao in \cite{LLZ},  is  a generalization of the Shen-Larsson  functor of type $(\huaW_n,\gl_n)$, which gives a class of new simple modules over $\huaW_n$. This class of simple $\huaW_n$-modules was used in the classification of simple  $\huaW_n$-modules that are finitely generated as modules over its Cartan subalgebra (see \cite{GLLZ}).
}
\end{rmk}

Next we take $\g=\C$, the one-dimensional trivial Lie algebra.  Let $p=(p_1,p_2,\cdots,p_n)\in \C[t_1^{\pm1}]\times  \C[t_2^{\pm1}]\times\cdots\times \C[t_n^{\pm1}],  q\in\C$.
Similar to the automorphism  $\sigma_b$ in Section 2 of  \cite{TZ}, we can easily see that
the linear map
$$\aligned  &\huaW_n\to \huaW_n\ltimes_{\Id} A_n,\\
&x^rd_i\mapsto x^r(d_i+p_i)+q r_ix^r,
\endaligned$$
is a Lie algebra homomorphism.
By Theorem \ref{twitst-iso} we see that the linear map
$$\aligned  H_{p,q}: &\huaW_n\to  \g\otimes A_n \cong A_n,\\
&
x^rd_i\mapsto (p_i+q r_i)x^r,\quad\forall r\in \Z^n, 1\leq i\leq n,
\endaligned$$is a crossed homomorphism from $\huaW_n$ to $A_n$. In fact, $H_{p,q}\in \Der_{\C}(\huaW_n,A_n)$. By Lemma \ref{CH-W} and Corollary \ref{twisting-functors}, we obtain the following result.

\begin{cor}\label{cor-5.6}
We have  a functor $F_{p,q} :\WRep_{\C}(\huaW_n)\to \WRep_{\C}(\huaW_n)$  defined by
$$
F_{p,q}({M};\rho)=({M} ;{\rho}\circ\iota_{H_{p,q}}),\quad\forall({M};\rho)\in\WRep_{\C}(\huaW_n).
$$
\end{cor}

\begin{rmk}{   By forgetting the $A_n$-module structure, the corresponding functor $F_{p,q}$ is just the twisting functor in the $\huaW_n$-module category introduced  in \cite{LGZ,LZ,TZ},  where a lot of new simple modules were obtained over the Virasoro algebra and $\huaW_n$.
}

\end{rmk}

\subsection{Shen  functors of  divergence zero type}

In this section we assume that $n\ge 2$.
Let us recall the divergence map $ \diver: \huaW_n\to A_n$ with $
x^rd_i\mapsto r_ix^r,$  for all $r\in \Z^n.$
It is well-known that
$$
\huaS_n=\{w\in \huaW_n\mid {\rm div}(w)=0\}
$$
  is a Lie subalgebra of $\huaW_n$, called the  Lie algebra of divergence
zero vector fields on an $n$-dimensional torus.  Let $d_{ij}(r)=r_jx^rd_i-r_ix^rd_j$. Then
$$
\huaS_n={\rm  span}_{\C}\{d_i, d_{ij}(r)\mid i,j=1,2\cdots,n\}.
$$
with the  Lie bracket
\begin{eqnarray*}
~[d_k,d_{ij}(r)]_{\huaW_n}&=&r_kd_{ij}(r),\\
~[d_{ij}(r),d_{pq}(s)]_{\huaW_n}&=&r_js_pd_{iq}(r+s)-r_js_qd_{ip}(r+s)-r_is_pd_{jq}(r+s)+r_is_qd_{jp}(r+s),
\end{eqnarray*}
for $r,s\in\Z^N, i,j,p,q=1,\cdots,n$.

Note that $\huaS_n$ is not a Lie-Rinehart subalgebra since $\huaS_n$ is not an $A_n$-module. It is straightforward to  see that $(A_n,\huaS_n,[\cdot,\cdot]_{\huaW_n},\Id)$ is a Leibniz pair.

Recall that   $\sln_n$ is the Lie subalgebra of $\gl_n$ consisting
of all traceless complex matrices.
The restriction $H|_{\huaS_n}$ of the crossed homomorphism $H$ in Lemma \ref{CH-W} is  a crossed homomorphism from
$\huaS_n$ to $\sln_n\otimes A_n$.
By Corollary \ref{Shen-Larsson-functors'}, we obtain the following result.

\begin{cor}
We have a functor $\huaF_{{H}}^{A_n}:\Rep_{\C}(\sln_n)\to \ARep_{\C}(\huaS_n)$  defined by
$$
\huaF_{H}^{A_n}(V;\theta)=(V\ot_\C{A_n}; ({\Id}\boxplus \theta)\circ\iota_{H}),\quad\forall(V;\theta)\in \Rep_{\C}(\sln_n).
$$
\end{cor}

\begin{rmk}{  The functor $\huaF_{H}^{A_n}$ is the well-known Shen-Larsson  functor of type $(\huaS_n,\sln_n)$, introduced  by
Shen  over polynomial algebras \cite{Sh} and further studied in \cite{BT,T}  over Laurent polynomial algebras. }
\end{rmk}

Let $(P;\rho)$ be a  representation of $\huaA_n$. It follows that $(P;\rho|_{\huaS_n})$ is an admissible representation of $\huaS_n$ since $\huaS_n\subset \huaA_n$. By Theorem \ref{bifunctor'}, we obtain the following result.

\begin{cor}
We have a functor  $\huaF_{{H}}^P:\Rep_{\C}(\sln_n)\to \ARep_{\C}(\huaS_n)$  defined by
$$
\huaF_{H}^P(V;\theta)=(V\ot_\C{P}; ({\rho|_{\huaS_n}}\boxplus \theta)\circ\iota_{H}),\quad\forall(V;\theta)\in \Rep_{\C}(\sln_n).
$$

\end{cor}

\begin{rmk}{  The functor $\huaF_{H}^P$  was introduced in   \cite{DGYZ} which is a generalization of the Shen-Larsson  functor of type $(\huaS_n,\sln_n)$, to give a class of new simple modules over $\huaS_n$.
}
\end{rmk}

\subsection{Shen  functors of Hamiltonian type}\label{hh}
For $r\in\Z^{2n}$, let
$$
h(r) = \sum_{i=1}^n(r_{n+i}x^r\partial_{i}-r_{i}x^r\partial_{n+i})\in \huaW_{2n}.
$$
It is well-known that
$
\huaH_n={\rm Span}_\C\{h(r)\mid r\in \Z^{2n}\}
$  is a Lie  subalgebra of $\huaW_{2n}$, with
$$
[h(r),h(s)]_{\huaW_{2n}}=\sum_{i=1}^n(r_{n+i}s_i-s_{n+i}r_{i})h(r+s),\quad\forall r,s\in \Z^{2n}.$$
This  Lie algebra $\huaH_n$ is called the Lie algebra of Hamiltonian vector fields on a $2n$-dimensional torus.
Note that $\huaH_n$ is not a Lie-Rinehart algebra  since $\huaH_n$ is not an $A_{2n}$-module.  It is straightforward to  see that $(A_n,\huaH_n,[\cdot,\cdot]_{\huaW_{2n}},\Id)$ is a Leibniz pair.

Let  $\mathfrak{sp}_{2n}$ be the Lie subalgebra of $\gl_{2n}$ consisting
of all symplectic matrices.
The restriction $H|_{\huaH_n}$ of the crossed homomorphism $H$ in Lemma \ref{CH-W} is  a linear map
$\huaH_n\to \mathfrak{sp}_{2n}\otimes A_{2n}$ given by
$$
H(h(r))=\begin{pmatrix}
  r_1r_{n+1}&\cdots&r_1r_{2n}&-r_1r_{1}&\cdots&-r_1r_{n}\\
  \vdots&\cdots&\vdots&\vdots&\cdots&\vdots\\
  r_nr_{n+1}&\cdots&r_nr_{2n}&-r_nr_{1}&\cdots&-r_nr_{n}\\
  r_{n+1}r_{n+1}&\cdots&r_{n+1}r_{2n}&-r_{n+1}r_{1}&\cdots&-r_{n+1}r_{n}\\
  \vdots&\cdots&\vdots&\vdots&\cdots&\vdots\\
  r_{2n}r_{n+1}&\cdots&r_{2n}r_{2n}&-r_{2n}r_{1}&\cdots&-r_{2n}r_{n}\\
\end{pmatrix}\ot x^r\in  \mathfrak{sp}_{2n}\otimes  A_{2n},
$$ which is certainly a crossed homomorphism from $\huaH_n$ to $\mathfrak{sp}_{2n}\otimes A_{2n}$.
By Corollary \ref{Shen-Larsson-functors'}, we obtain the following result.

\begin{cor}\label{cor-5.12}
We have a functor $\huaF_{{H}}^{A_{2n}}:\Rep_{\C}(\mathfrak{sp}_{2n})\to \ARep_{\C}(\huaH_n)$  defined by
$$
\huaF_{H}^{A_{2n}}(V;\theta)=(V\ot_\C{A_{2n}}; ({\Id}\boxplus \theta)\circ\iota_{H}),\quad\forall(V;\theta)\in \Rep_{\C}(\mathfrak{sp}_{2n}).
$$

\end{cor}
\begin{rmk}{   $\huaF_{H}^{A_{2n}}$ is the well-known Shen-Larsson  functor of type $(\huaH_n,\mathfrak{sp}_{2n})$, introduced  by
Shen  over polynomial algebras \cite{Sh}. Note that there are no results for $\huaH_n$ similar to those in \cite{E1, GZ, LZ0}.}
\end{rmk}

\subsection{Actions of monoidal categories  for generalized  Cartan type }

Let $A$ be a commutative associative $\C$-algebra, and let $\Delta$ be a nonzero $\C$-vector space of commuting $\C$-derivations of $A$. Let us first recall the  construction of the generalized Witt algebras from \cite{Pa}. The tensor product $A\Delta:=A\otimes_{\C}\Delta$ acts on $A$ by
$$
a\ot\partial: x\mapsto a\partial(x),\quad  a,x\in A, \partial\in \Delta.
$$
Since $A$ is commutative, this gives rise to a linear transformation
$
\al: A\Delta\to \Der_{\C}(A).
$
Define a bracket $[\cdot,\cdot]_{A\Delta}$ on $A\Delta$ by
$$
[a\partial,b\delta]_{A\Delta}=a\partial(b)\delta- b\delta(a)\partial,\quad\forall a,b\in A, \partial,\delta\in\Delta,
$$
which gives a Lie algebra structure on $A\Delta$. Then $\al$ is
clearly an action of $A\Delta$ on the commutative Lie algebra $A$.  Assume that $\dim_{\C}\Delta <\infty$. Then there are   $\partial_1,\cdots,\partial_n\in \Delta$ such that $A\Delta$ is a free $A$-module with basis  $\{\partial_1,\cdots,\partial_n\}$ (see \cite{Z}). We denote this Lie algebra by $\huaW_{n}(A,\Delta)$. Note that $(A,\huaW_{n}(A,\Delta),[\cdot,\cdot]_{_{A\Delta}},\al)$ is a Lie-Rinehart algebra.

Now we have a generalization of Lemma \ref{CH-W}.

\begin{lem}\label{CH-GW}
The linear map $H: \huaW_n(A,\Delta)\to  \gl_n \ot A$  defined by
$$
H\left(\sum_{i=1}^n a_i\partial_i\right)=\sum_{i=1}^n\sum_{j=1}^n E_{ij}\ot\partial_i (a_j),\quad a_i\in A
$$
 is a crossed homomorphism from $\huaW_n(A,\Delta)$ to $\gl_n \ot A$.
\end{lem}
\begin{proof}
It is straightforward but tedious to verify the above formula. We omit the details.
\end{proof}

Similar to  Corollary \ref{cor:trivialm}, by Lemma \ref{CH-GW}  and Theorem \ref{bifunctor} we obtain the following result.

 \begin{cor}\label{cor:gw}
We have a functor $F_{{H}}^{A}:\Rep_{\C}(\gl_n)\to \WRep_{\C}(\huaW_n(A,\Delta))$  defined by
$$
F_{H}^A(V;\theta)=(V\ot_\C{A}; ({\alpha}\boxplus \theta)\circ\iota_{H}),\quad\forall(V;\theta)\in \Rep_{\C}(\gl_n).
$$
 \end{cor}

\begin{rmk}{
\begin{enumerate}
\item If $A=\C[x_1^{\pm1},\cdots,x_n^{\pm1}]$ and  $\Delta={\rm Span}_\C\{x_1\frac{\partial}{\partial x_1},\cdots, x_n\frac{\partial}{\partial x_n}\}$, $\huaW_n(A,\Delta)$ is the standard Witt algebra $\huaW_n$ and the corresponding $F_{H}^A$  is the Shen-Larsson  functor of type $(\huaW_n,\gl_n)$.

 \item If $A=\C[x_1^{\pm1},\cdots,x_n^{\pm1}]$ and  $\bar\Delta={\rm Span}_\C\{\frac{\partial}{\partial x_1},\cdots, \frac{\partial}{\partial x_n}\}$, $\huaW_n(A,\bar\Delta)$ is also the standard Witt algebra $\huaW_n$. However,  the corresponding Shen-Larsson  functor $\bar F_{H}^A$  is different from the standard $F_{H}^A$ except on the category of finite-dimensional $\gl_n$-modules. This was pointed out by Liu, Lu and Zhao in \cite{LLZ}.

 \item If $A$ is taken to be a polynomial algebra with finitely many variables $x_i$ together with some $x_i^{-1}$, $\Delta$ to be some mixed differential operators w.r.t $x_i$,  the Lie algebra  $\huaW_n(A,\Delta)$ was
introduced by Xu \cite{Xu}.  The corresponding Shen-Larsson  functor $F_{H}^A$ was introduced and studied by  Zhao \cite{Zhy}, generalizing Rao's results in \cite{E1}.

\item Under certain finite conditions,  the functor $F_{H}^A$  is the Shen-Larsson  functor $\huaW_n(A,\Delta)$ introduced and studied by Skryabin in \cite{S2}.

 \item Let $A$ be the coordinate ring of an irreducible affine variety and  $\Delta$ certain subalgebra of $\Der (A)$. The corresponding Shen-Larsson  functor $F_{H}^A$  have been introduced and studied in  \cite{BF,BFN} to give new simple modules over $\huaW_n(A,\Delta)$.
\end{enumerate}
}
\end{rmk}

Now let us define the divergence map $\diver: \huaW_n(A,\Delta)\to A$ to be the $\C$-linear
extension of
$$
\diver(a\partial) = \partial(a),\quad \forall a\in A, \partial\in \Delta.
$$
Let
$
\huaS_n(A,\Delta)=\{w\in A\Delta\mid {\rm div}(w)=0\}.
$
Then $\huaS_n(A,\Delta)$ is a Lie subalgebra of $\huaW_n(A,\Delta)$, see \cite{BP} for more details. If $A=\C[x_1^{\pm1},\cdots,x_n^{\pm1}]$ and  $\Delta={\rm Span}_\C\{x_1\frac{\partial}{\partial x_1},\cdots, x_n\frac{\partial}{\partial x_n}\}$, $\huaS_n(A,\Delta)$ is the Lie algebra $S_n$ of divergence
zero vector fields on an $n$-dimensional torus.

Note that $\huaS_n(A,\Delta)$ is not a Lie-Rinehart subalgebra since $\huaS_n(A,\Delta)$ is not an $A$-module. It is straightforward to  see that $(A,\huaS_n(A,\Delta),[\cdot,\cdot]_{A\Delta},\Id)$ is a Leibniz pair.

It is clear that $H|_{\huaS_n(A,\Delta)}$ is a crossed homomorphism from $\huaS_n(A,\Delta)$ to $\sln_n \ot A$.
Similar to  Corollary \ref{cor-5.6}, by Lemma \ref{CH-GW}  and Corollary \ref{Shen-Larsson-functors'}, we obtain the following result.

\begin{cor}\label{cor:gs}
We have  a functor $\huaF_{{H}}^{A}:\Rep_{\C}(\sln_n)\to \ARep_{\C}(\huaS_n(A,\Delta))$  defined by
$$
\huaF_{H}^A(V;\theta)=(V\ot_\C{A}; ({\alpha}\boxplus \theta)\circ\iota_{H}),\quad\forall(V;\theta)\in \Rep_{\C}(\sln_n).
$$
\end{cor}

Now let us define a map $D: A\to \huaW_{2n}(A,\Delta) $ to be the linear
extension of
$$
D(a) = \sum_{i=1}^n(\partial_{i}(a)\partial_{n+i}-\partial_{n+i}(a)\partial_{i}),\quad \forall a\in A.
$$
Let
$
\huaH_n(A,\Delta)=\{D(a)\mid a\in A\}.
$
Then $\huaH_n(A,\Delta)$ is a Lie  subalgebra of $\huaW_{2n}(A,\Delta)$, with
$$
[D(a),D(b)]_{A\Delta}=D\left(\sum_{i=1}^n(\partial_{i}(a)\partial_{n+i}(b)-\partial_{n+i}(a)\partial_{i}(b))\right),\quad\forall a,b\in A.
$$
If $A=\C[x_1^{\pm1},\cdots,x_{2n}^{\pm1}]$ and  $\Delta={\rm Span}_\C\{x_1\frac{\partial}{\partial x_1},\cdots, x_{2n}\frac{\partial}{\partial x_{2n}}\}$, then $\huaH_n(A,\Delta)$ is the Lie algebra  of Hamiltonian vector fields on a $2n$-dimensional torus.

 Note that $\huaH_n(A,\Delta)$ is not a Lie-Rinehart algebra  since $\huaH_n(A,\Delta)$ is not an $A$-module.  It is straightforward to  see that $(A,\huaH_n(A,\Delta),[\cdot,\cdot]_{A\Delta},\Id)$ is a Leibniz pair.

The restriction $H|_{\huaH_n(A,\Delta)}$ of the crossed homomorphism $H$ in Lemma \ref{CH-GW} is  a crossed homomorphism from $\huaH_n(A,\Delta)$ to $\mathfrak{sp}_{2n}\otimes A$. Similar to  Corollary \ref{cor-5.12}, by Lemma \ref{CH-GW}  and Corollary \ref{Shen-Larsson-functors'}, we obtain the following result.

\begin{cor}\label{cor:gh}
We have  a functor $\huaF_{{H}}^{A}:\Rep_{\C}(\mathfrak{sp}_{2n})\to \ARep_{\C}(\huaH_n(A,\Delta))$  defined by
$$
\huaF_{H}^A(V;\theta)=(V\ot_\C{A}; ({\alpha}\boxplus \theta)\circ\iota_{H}),\quad\forall(V;\theta)\in \Rep_{\C}(\mathfrak{sp}_{2n}).
$$
\end{cor}

\section{Deformation and cohomologies  of crossed homomorphisms}\label{sec:deformation}

In this section, first we give the Maurer-Cartan characterization of crossed homomorphisms of Lie algebras. In particular, we give the differential graded Lie algebra that control deformations of crossed homomorphisms.  Then we define the cohomology groups of crossed homomorphisms, which can be applied   to study linear deformations of crossed homomorphisms.

\subsection{The differential graded Lie algebra controlling deformations}

\begin{defi}\label{1.1} A {\bf differential graded Lie algebra}
$(\g,[\cdot,\cdot],d)$ is   a
$\Z$-graded vector space
$\g=\oplus_{i\in \Z}\g_i$ together with a bilinear bracket
$[\cdot,\cdot]\colon \g\otimes \g\to \g$ and a linear map $d\colon \g\to \g$ satisfying
the following conditions:\begin{itemize}

\item  $[\g_i,\g_j]\subset\g_{i+j}$ and
$[a,b]=-(-1)^{\overline{a}\overline{b}}[b,a]$ for every $a,b$ homogeneous.

\item  Every $a,b,c$ homogeneous satisfy the Jacobi identity
\[ [a,[b,c]]=[[a,b],c]+(-1)^{\overline{a}\overline{b}}[b,[a,c]].\]

\item  $d(\g_i)\subset \g_{i+1}$, $d\circ d=0$ and
$d[a,b]=[da,b]+(-1)^{\overline{a}}[a,db]$. The map $d$ is called the
{\bf differential} of $\g$.
\end{itemize}
We have used the notation $\bar a=i$ if $a\in \g_i$.
\end{defi}

\begin{defi}\label{MCE}{\rm (\cite{LV})}
  Let $(\g=\oplus_{k\in\mathbb Z}\g_k,[\cdot,\cdot],d)$ be a differential graded Lie algebra.  A degree $1$ element $\theta\in\g_1$ is called a {\bf Maurer-Cartan element} of $\g$ if it
  satisfies the following {\bf Maurer-Cartan equation:}
\vspace{-.3cm}
  \begin{equation}
  d \theta+\half[\theta,\theta]=0.
  \label{eq:mce}
  \end{equation}
  \end{defi}

\begin{pro}\label{deformation-MCE}{\rm (\cite{LV})}
Let $(\g=\oplus_{k\in\mathbb Z}\g_k,[\cdot,\cdot],d)$ be a differential graded Lie algebra and let $\mu\in \g_1$ be a Maurer-Cartan element. Then the map
$$ d_\mu: \g \longrightarrow \g,\quad \ d_\mu(x):=d(x)+[\mu, x], \quad \forall x\in \g,$$
is a differential on the graded Lie algebra $(\g,[\cdot,\cdot])$. For any $v\in \g_1$, the sum $\mu+v$ is a
Maurer-Cartan element of the differential graded Lie algebra $(\g,
[\cdot,\cdot],d)$ if and only if $v$ is a Maurer-Cartan element of the differential graded Lie algebra $(\g, [\cdot,\cdot], d_\mu)$.
\end{pro}

Let $(\g,[\cdot,\cdot]_\g)$ and $(\h,[\cdot,\cdot]_\h)$ be Lie algebras and  $\rho:\g\lon\Der(\h)$ an action of $\g$ on $\h$. Consider the graded vector space
$$\huaC^*(\g,\h):=\oplus_{k\geq 0}\Hom(\wedge^{k}\g,\h).$$
Define $d:\Hom(\wedge^m\g,\h)\lon\Hom(\wedge^{m+1}\g,\h)$ by
\begin{eqnarray}
\label{graded-der}(df) ( x_1,\cdots, x_{m+1} )
&=&\sum_{i=1}^{m+1}(-1)^{m+i}\rho(x_i)f(x_1,\cdots,\hat{x}_i,\cdots,x_{m+1})\\
\nonumber&&+\sum_{1\le i< j\le m+1}(-1)^{m+i+j-1}f([x_i,x_j]_{\g},x_1,\cdots,\hat{x}_i,\cdots,\hat{x}_{j},\cdots,x_{m+1}),
\end{eqnarray}
 for all $f\in \Hom(\wedge^m\g,\h)$. Define a skew-symmetric bracket operation $\Courant{\cdot,\cdot}: \Hom(\wedge^m\g,\h)\times \Hom(\wedge^n\g,\h)\longrightarrow \Hom(\wedge^{m+n}\g,\h)$ by
\begin{eqnarray}
&&\nonumber\Courant{f_1,f_2}(x_1,x_2,\cdots,x_{m+n})\\
\label{graded-Lie}&=&(-1)^{mn+1}\sum_{\sigma\in \mathbb S_{(m,n)}}(-1)^{\sigma}[f_1(x_{\sigma(1)},\cdots,x_{\sigma(m)}),f_2(x_{\sigma(m+1)},\cdots,x_{\sigma(m+n)})]_\h
\vspace{-.1cm}
\end{eqnarray}
for all $f_1\in\Hom(\wedge^m\g,\h)$ and $f_2\in\Hom(\wedge^n\g,\h)$. Here $\mathbb S_{(m,n)}$   denotes   the set of all $(m,n)$-shuffles.

Note that for all $u,v\in\h$, $\Courant{u,v}=-[u,v]_\h$.

\begin{pro}\label{crossed-homo-MC}
  With the above notations, $(\huaC^*(\g,\h),\Courant{\cdot,\cdot},d)$ is a differential graded Lie algebra. Its Maurer-Cartan elements are precisely   crossed homomorphisms from $\g$ to $\h$ with respect to the action $\rho$.
\end{pro}

\begin{proof}
In short, the graded Lie algebra $(\huaC^*(\g,\h),\Courant{\cdot,\cdot})$ is obtained via the derived bracket \cite{Kosmann-Schwarzbach,Vor}. In fact, the Nijenhuis-Richardson bracket $[\cdot,\cdot]_{\NR}$ associated to the direct sum vector space $\g\oplus V$ gives rise to a graded Lie algebra $(\oplus_{k\geq 0}\Hom(\wedge^k(\g\oplus \h),\g\oplus \h),[\cdot,\cdot]_{\NR})$. Obviously $\oplus_{k\geq 0}\Hom(\wedge^k\g,\h)$ is an abelian subalgebra. We denote the Lie brackets $[\cdot,\cdot]_\g$ and $[\cdot,\cdot]_\h$ by $\mu_\g$ and $\mu_\h$ respectively. Since $\rho$ is an  action of the Lie algebra $(\g,[\cdot,\cdot]_\g)$. We deduce that $\mu_\g+\rho$ is a semidirect product Lie algebra structure on $\g\oplus\h$. Thus $\mu_\g+\rho$ and $\mu_\h$ are Maurer-Cartan elements of the graded Lie algebra $(\huaC^*(\g\oplus\h,\g\oplus\h),[\cdot,\cdot]_{\NR})$. Define a differential $d_{\mu_\h}$ on $(\huaC^*(\g\oplus\h,\g\oplus\h),[\cdot,\cdot]_{\NR})$ via
$$
  d_{\mu_\h}:=[\mu_\h,\cdot]_{\NR}.
$$
Further, we define the derived bracket on the graded vector space $\oplus_{k\geq 0}\Hom(\wedge^k\g,\h)$ by
\vspace{-.1cm}
$$
  \Courant{f_1,f_2}:=(-1)^{m-1}[[\mu_\h,f_1]_{\NR},f_2]_{\NR},\quad\forall f_1\in\Hom(\wedge^m\g,\h),~f_2\in\Hom(\wedge^n\g,\h),
\vspace{-.1cm}
$$
which is exactly the bracket given by \eqref{graded-Lie}. By $[\mu_\h,\mu_\h]_{\NR}=0,$ we deduce that $(\huaC^*(\g,\h),\Courant{\cdot,\cdot})$ is a graded Lie algebra.

Moreover, by $\Img\rho\subset\Der(\h)$, we have $[\mu_\g+\rho,\mu_\h]_{\NR}=0.$ We define a linear map $d=:[\mu_\g+\rho,\cdot]_{\NR}$ on the graded space $\huaC^*(\g\oplus\h,\g\oplus\h)$. We obtain that $d$ is closed on the subspace $\oplus_{k\geq 0}\Hom(\wedge^k\g,\h)$, and is given by \eqref{graded-der}.

By $[\mu_\g+\rho,\mu_\g+\rho]_{\NR}=0$, we obtain that $d\circ d=0.$ Moreover, by $[\mu_\g+\rho,\mu_\h]_{\NR}=0$, we deduce that $d$ is a derivation of $(\huaC^*(\g,\h),\Courant{\cdot,\cdot})$. Therefore,  $(\huaC^*(\g,\h),\Courant{\cdot,\cdot},d)$ is a differential graded Lie algebra.

Finally, for a degree one element ${H}\in \Hom(\g,\h)$, we have
\begin{eqnarray*}
\big(dH+\frac{1}{2}\Courant{{H},{H}}\big)(x,y)=\rho(x)({H} y)-\rho(y)({H} x)-{H}[x,y]_\g+[{H} x,{H} y]_\h.
\end{eqnarray*}
Thus, Maurer-Cartan elements are precisely crossed homomorphisms from  $(\g,[\cdot,\cdot]_\g)$ to $(\h,[\cdot,\cdot]_\h)$ with respect to the action $\rho$. The proof is finished.
\end{proof}

\emptycomment{
\begin{proof}
By Lemma \ref{twilled-DGLA-concrete}, we get the dgLa structure on $\huaC^*(\g,\h)=\oplus_{k\geq 0}\Hom(\wedge^{k}\g,\h)$ as above. For ${H}\in C^2(\g,\h)$, we have
\begin{eqnarray*}
\Big(d({H})+\frac{1}{2}[[{H},{H}]]\Big)(x_1,x_2)=\rho(x_1)({H} x_2)-\rho(x_2)({H} x_1)-{H}[x_1,x_2]_\g+[{H} x_1,{H} x_2]_\h.
\end{eqnarray*}
Thus, the Maurer-Cartan elements are precisely crossed homomorphism from Lie algebra $(\g,[\cdot,\cdot]_\g)$ to $(\h,[\cdot,\cdot]_\h)$. The proof is finished.
\end{proof}
}

Let ${H}:\g\longrightarrow\h$ be a crossed homomorphism with respect to the action $\rho$. Since ${H}$ is a Maurer-Cartan element of the differential graded Lie algebra $(\huaC^*(\g,\h),\Courant{\cdot,\cdot},d)$ by Proposition~\ref{crossed-homo-MC}, it follows from Proposition~\ref{deformation-MCE} that
$d_{{H}}:=d+\Courant{{H},\cdot}$
 is a graded derivation on the graded Lie
algebra $(\huaC^*(\g,\h),\Courant{\cdot,\cdot})$ satisfying $d^2_{{H}}=0$.
  Therefore, $(\huaC^*(\g,\h),\Courant{\cdot,\cdot},d_{{H}})$ is a differential graded Lie algebra.
This differential graded Lie algebra can control deformations of crossed homomorphisms.  We have obtained the following result.

\begin{thm}\label{thm:deformation}
Let ${H}:\g\longrightarrow\h$ be a crossed homomorphism with respect to the action $\rho$. For a
linear map ${H}':\g\longrightarrow\h$, then  ${H}+{H}'$ is still
a crossed homomorphism from $\g$ to $\h$ with respect to the action $\rho$ if and only if ${H}'$ is a Maurer-Cartan
element of the differential graded Lie algebra
$(\huaC^*(\g,\h),\Courant{\cdot,\cdot},d_{{H}})$.
\end{thm}

\subsection{Cohomologies of crossed homomorphisms}\label{sec:deformation2}

In this subsection, we define cohomologies of a crossed homomorphism, which can be used to study linear deformations in Section \ref{sec:ld}.

Recall that $\rho_H$ defined by \eqref{eq:newrep} is a representation of $\g$ on $\h$.
Let $d_{\rho_{{H}}}: \Hom(\wedge^k\g,\h)\longrightarrow
\Hom(\wedge^{k+1}\g,\h)$ be the corresponding Chevalley-Eilenberg
coboundary operator. More precisely, for all $f\in
\Hom(\wedge^k\g,\h)$ and $x_1,\cdots,x_{k+1}\in \g$, we have
\begin{eqnarray}
&& d_{\rho_{{H}}} f(x_1,\cdots,x_{k+1})\notag\\
&=&\sum_{i=1}^{k+1}(-1)^{i+1}\rho(x_i)f(x_1,\cdots,\hat{x}_i,\cdots, x_{k+1})+\sum_{i=1}^{k+1}(-1)^{i+1}[{H} x_i,f(x_1,\cdots,\hat{x}_i,\cdots, x_{k+1})]_\h\\
&&+\sum_{1\le i<j\le k+1}(-1)^{i+j}f([x_i,x_j]_\g,x_1,\cdots,\hat{x}_i,\cdots,\hat{x}_j,\cdots, x_{k+1}). \notag
\end{eqnarray}
It is obvious that $u\in\h$ is closed if and only if
         $
           \rho(x)u+[{H} x,u]_\h=0
         $
         for all $x\in\g$,
               and   $f\in \Hom(\g,\h)$  is closed  if and only if
               $$
              \rho(x_1)f(x_2)-\rho(x_2)f(x_1)+[{H} x_1,f(x_2)]_\h-[{H} x_2,f(x_1)]_\h-f([x_1,x_2]_\g)=0,\quad\forall x_1,x_2\in\g.
               $$

\begin{defi}
Let ${H}:\g\longrightarrow\h$ be a crossed homomorphism with respect to the action $\rho$. Denote by $\huaC^k(\g,\h)=\Hom(\wedge^k\g,\h)$ and $(\huaC^*(\g,\h)=\oplus _{k\geq 0}\huaC^k(\g,\h),d_{\rho_{{H}}})$ the above cochain complex. Denote the set of $k$-cocycles by $\huaZ^k(\g,\h)$ and the set of $k$-coboundaries by $\huaB^k(\g,\h)$. Denote by
  \begin{equation}
  \huaH^k(\g,\h)=\huaZ^k(\g,\h)/\huaB^k(\g,\h), \quad k \geq 0,
  \label{eq:ocoh}
  \end{equation}
the $k$-th cohomology group which will be taken to be the {\bf $k$-th cohomology group for the crossed homomorphism $ H$}.
\label{de:opcoh}
\end{defi}

Comparing the coboundary operators $d_{\rho_{{H}}}$ given above and the
operators $d_{{H}}=d+\Courant{{H},\cdot}$ defined by the Maurer-Cartan element ${H}$, we have

\begin{pro}\label{pro:danddT}
 Let ${H}:\g\longrightarrow\h$ be a crossed homomorphism. Then we have
 $$
 d_{\rho_{{H}}}f=(-1)^{k-1}d_{{H}}f,\quad \forall f\in \Hom(\wedge^k\g,\h).
 $$
\end{pro}

\begin{proof}
Indeed, for all $x_1,x_2,\cdots,x_{k+1}\in \g$ and $f\in \Hom(\wedge^k\g,\h)$ , we have
\begin{eqnarray*}
(-1)^{k-1}(d_{{H}}f)(x_1,x_2,\cdots,x_{k+1}) &=&(-1)^{k-1}(df+\Courant{{H},f})(x_1,\cdots,x_{k+1})\\
&=&\sum_{i=1}^{i+1}(-1)^{i+1}\rho(x_i)f(x_1,\cdots,\hat{x}_i,\cdots,x_{k+1})\\
&&+\sum_{1\le i< j\le k+1}(-1)^{i+j}f([x_i,x_j]_{\g},x_1,\cdots,\hat{x}_i,\cdots,\hat{x}_{j},\cdots,x_{k+1})\\
&&+(-1)^{k-1}(-1)^{k+1}\sum_{\sigma\in \mathbb S_{(1,k)}}(-1)^{\sigma}[{H} x_{\sigma(1)},f(x_{\sigma(2)},\cdots,x_{\sigma(k+1)})]_\h\\
&=&\sum_{i=1}^{i+1}(-1)^{i+1}\rho(x_i)f(x_1,\cdots,\hat{x}_i,\cdots,x_{k+1})\\&&+\sum_{1\le i< j\le k+1}(-1)^{i+j}f([x_i,x_j]_{\g},x_1,\cdots,\hat{x}_i,\cdots,\hat{x}_{j},\cdots,x_{k+1})\\
&&+\sum_{i=1}^{k+1}(-1)^{i-1}[{H} x_{i},f(x_{1},\cdots,\hat{x}_i,\cdots,x_{k+1})]_\h\\
&=&(d_{\rho_{{H}}}f)(x_1,x_2,\cdots,x_{k+1}),
\end{eqnarray*}
which implies that $d_{\rho_{{H}}}f=(-1)^{k-1}d_{{H}}f$.
\end{proof}

At the end of this section, we show that certain homomorphisms between crossed homomorphisms induce homomorphisms between the corresponding cohomology groups.  Let ${H}$ and $\widetilde{H}$ be two crossed homomorphisms from $\g$ to $\h$ with respect to the action $\rho$, and $(\phi_\g,\phi_\h)$ a homomorphism   from $\widetilde{H}$ to ${H}$ in which $\phi_\g$ is invertible.  For all $k\geq 0$, define
 \begin{eqnarray*}
\Phi:&&\Hom(\wedge^k\g,\h)\lon \Hom(\wedge^k\g,\h)\\
 &&f\mapsto \phi_\h\circ f\circ (\phi_\g^{-1})^{\otimes k}.
\end{eqnarray*}

\begin{thm}
   Let ${H}$ and $\widetilde{H}$ be two crossed homomorphisms from $\g$ to $\h$ with respect to the action $\rho$ of $\g$ on $\h$, and $(\phi_\g,\phi_\h)$ a homomorphism   from $\widetilde{H}$ to ${H}$ in which $\phi_\g$ is invertible.  Then the above $\Phi$ is a cochain map from the cochain complex $(\huaC^*(\g,\h),d_{\rho_{\widetilde{H}}})$ to $(\huaC^*(\g,\h),d_{\rho_{{H}}})$. Consequently, $\Phi$ induces a homomorphism $\Phi_*: \widetilde{\huaH}^k(\g,\h)\lon  \huaH^k(\g,\h)$ between corresponding cohomology groups.
\end{thm}

\begin{proof} By the fact that $(\phi_\g,\phi_\h)$ is a homomorphism   from $\widetilde{H}$ to ${H}$, we have
 \begin{eqnarray*}
    && (\Phi (d_{\rho_{\widetilde{{H}}}}f))(x_1,\cdots,x_{k+1})=\phi_\h (d_{\rho_{\widetilde{{H}}}}f)(\phi_\g^{-1} (x_1),\cdots,\phi_\g^{-1} (x_{k+1}))\\
     &=&\sum_{i=1}^{i+1}(-1)^{i+1}\phi_\h\rho(\phi_\g^{-1}(x_i))f(\phi_\g^{-1} (x_1),\cdots,\hat{x}_i,\cdots,\phi_\g^{-1} (x_{k+1}))\\
     &&+\sum_{1\le i< j\le k+1}(-1)^{i+j}\phi_\h f([\phi_\g^{-1} (x_i),\phi_\g^{-1} (x_j)]_{\g},\phi_\g ^{-1} (x_1),\cdots,\hat{x}_i,\cdots,\hat{x}_{j},\cdots,\phi_\g^{-1} (x_{k+1}))\\
&&+\sum_{i=1}^{k+1}(-1)^{i+1}\phi_\h[\widetilde{{H}} \phi_\g^{-1} (x_{i}),f(\phi_\g^{-1} (x_{1}),\cdots,\hat{x}_i,\cdots,\phi_\g^{-1}( x_{k+1}))]_\h\\
&=&\sum_{i=1}^{i+1}(-1)^{i+1}\rho( x_i)\phi_\h f(\phi_\g^{-1} (x_1),\cdots,\hat{x}_i,\cdots,\phi_\g^{-1} (x_{k+1}))\\
     &&+\sum_{1\le i< j\le k+1}(-1)^{i+j}\phi_\h f(\phi_\g^{-1}[ x_i,x_j]_{\g},\phi_\g ^{-1} (x_1),\cdots,\hat{x}_i,\cdots,\hat{x}_{j},\cdots,\phi_\g^{-1} (x_{k+1}))\\
&&+\sum_{i=1}^{k+1}(-1)^{i+1}[{H} (x_{i}),\phi_\h f(\phi_\g^{-1} (x_{1}),\cdots,\hat{x}_i,\cdots,\phi_\g^{-1}( x_{k+1}))]_\h\\
&=& d_{\rho_{{H}}}\Phi (f)(x_1,\cdots,x_{k+1}),
 \end{eqnarray*}
 which implies that $\Phi$ is a cochain map.
\end{proof}

\begin{cor}
Let ${H}$ and $\widetilde{H}$ be two isomorphic crossed homomorphisms. Then the cohomology groups $\widetilde{\huaH}^k(\g,\h)$ and $  \huaH^k(\g,\h)$ are isomorphic for any $k\in\Z_+$.
\end{cor}

\subsection{Linear deformations of crossed homomorphisms}\label{sec:ld}

In this subsection, we study linear deformations of crossed homomorphisms using the cohomology theory introduced in   Section \ref{sec:deformation2}, and show that isomorphic linear deformations are identified  with the same class in the second cohomology group. We give the notion of a Nijenhuis element associated to a crossed homomorphism, which  gives rise to a trivial deformation.

\begin{defi}
Let ${H}:\g\longrightarrow\h$ be a crossed homomorphism with respect to the action $\rho$ and $\frkH:\g\longrightarrow\h$ a linear map. If  ${H}_t={H}+t\frkH$ is still a crossed homomorphism from $\g$ to $\h$ with respect to the action $\rho$ for all $t$, we say that $\frkH$ generates a {\bf (one-parameter) linear deformation} of the crossed homomorphism ${H}$.
    \end{defi}
It is direct to check that ${H}_t={H}+t\frkH$ is a linear deformation of a crossed homomorphism ${H}$ if and only if
 for any $x,y\in \g$,
\begin{eqnarray}
\label{deformation-1}\rho(x)\frkH y-\rho(y)\frkH x+[{H} x,\frkH y]_\h+[\frkH x,{H} y]_\h-\frkH[x,y]_\g&=&0,\\
\label{deformation-2}[\frkH x,\frkH y]_\h&=&0.
\end{eqnarray}
Note that Eq.~\eqref{deformation-1} means that $\frkH$ is a 1-cocycle of the crossed homomorphism $H$.

\begin{defi} Let ${H}$ be a crossed homomorphism from $\g$ to $\h$ with respect to the action $\rho$.
\begin{itemize}
\item[\rm(i)]Two linear deformations ${H}_t^1={H}+t\frkH_1$ and
${H}^2_t={H}+t\frkH_2$ are said to be {\bf equivalent} if there exists
an $x\in\g$ such that $({\Id}_\g+t\ad_x,{\Id}_\h+t\rho(x))$ is a
homomorphism from ${H}^2_t$ to ${H}^1_t$.

\item[\rm(ii)]A
linear deformation ${H}+t\frkH$ of a
crossed homomorphism ${H}$ is said to be {\bf trivial} if there exists
an $x\in\g$ such that $({\Id}_\g+t\ad_x,{\Id}_\h+t\rho(x))$ is a
homomorphism from ${H}_t$ to ${H}$.
\end{itemize}
\end{defi}

Let $({\Id}_\g+t\ad_x,{\Id}_\h+t\rho(x))$ be a homomorphism from
${H}^2_t$ to ${H}^1_t$. Then ${\Id}_\g+t\ad_x$ and ${\Id}_\h+t\rho(x)$ are Lie algebra
endomorphisms. Thus, we have
\begin{eqnarray*}
({\Id}_\g+t\ad_x)[y,z]_\g&=&[({\Id}_\g+t\ad_x)(y),({\Id}_\g+t\ad_x)(z)]_\g, \;\forall y,z\in \g,\\
({\Id}_\h+t\rho(x))[u,v]_\h&=&[({\Id}_\h+t\rho(x))(u),({\Id}_\h+t\rho(x))(v)]_\h, \;\forall u,v\in \h.
\end{eqnarray*}
which implies that $x$ satisfies
\begin{eqnarray}
\label{eq:Nij1}[[x,y]_\g,[x,z]_\g]_\g=0,\quad \forall y,z\in\g,\\
\label{eq:Nij2}[\rho(x)u,\rho(x)v]_\h=0,\quad \forall u,v\in\h.
\end{eqnarray}

Then by Eq.~\eqref{homo-1}, we get
$$
({H}+t\frkH_1)({\Id}_\g+t\ad_x)(y)=({\Id}_\h+t\rho(x))({H}+t\frkH_2)(y),\quad\forall y\in \g,
$$
which implies
\begin{eqnarray}
 (\frkH_2-\frkH_1)(y)&=&-\rho(y){H} x-[{H} y,{H} x]_\h,\label{eq:deforiso1} \\
  \frkH_1[x,y]_\g&=&\rho(x)(\frkH_2y), \quad \forall y\in \g.
  \label{eq:deforiso2}
\end{eqnarray}
Finally, Eq.~\eqref{homo-2} gives
$$
({\Id}_\h+t\rho(x))\rho(y)(u)=\rho(({\Id}_\g+t\ad_x)(y))({\Id}_\h+t\rho(x))(u),\quad \forall y\in\g, u\in \h,
$$
which implies that $x$ satisfies
\begin{equation}
  \rho([x,y]_\g)\rho(x)=0,\quad\forall y\in\g.\label{eq:Nij3}
\end{equation}

Note that Eq.~\eqref{eq:deforiso1} means that $\frkH_2-\frkH_1=d_{\rho_{{H}}} (-{H} x)$. Thus, we have

\begin{thm}\label{thm:iso3} Let ${H}$ be a crossed homomorphism from $\g$ to $\h$ with respect to the action $\rho$.
  If two linear deformations ${H}^1_t={H}+t\frkH_1$ and ${H}^2_t={H}+t\frkH_2$ are equivalent, then $\frkH_1$ and $\frkH_2$ are in the same cohomology class of $\huaH^1(\g,\h)=\huaZ^1(\g,\h)/\huaB^1(\g,\h)$ defined in
  Definition~\ref{de:opcoh}.
\end{thm}

\begin{defi}
Let ${H}$ be a crossed homomorphism from $\g$ to $\h$ with respect to the action $\rho$. An element $x\in\g$ is called a {\bf Nijenhuis element} associated to ${H}$ if $x$ satisfies Eqs.~\eqref{eq:Nij1}, \eqref{eq:Nij2}, \eqref{eq:Nij3} and the equation
      \begin{eqnarray}
         \rho(x)\Big(\rho(y){H} x+[{H} y,{H} x]_\h\Big)=0,\quad\forall y\in\g.
         \label{eq:Nij4}
        \end{eqnarray}
   Denote by $\Nij({H})$ the set of Nijenhuis elements associated to a crossed homomorphism ${H}$.
    \end{defi}

By Eqs.~\eqref{eq:Nij1}-\eqref{eq:Nij3}, it is obvious that a
trivial linear deformation gives rise to a
Nijenhuis element. The following result is in close analogue to
the fact that the differential of a Nijenhuis operator on a Lie
algebra generates a trivial linear
deformation of the Lie algebra~\cite{Do}, justifying the notion of
Nijenhuis elements.

\begin{thm}\label{thm:trivial}
   Let ${H}$ be a crossed homomorphism from $\g$ to $\h$ with respect to the action $\rho$. Then for any  $x\in \Nij({H})$, ${H}_t:={H}+t \frkH$ with $\frkH:=d_{\rho_{H}} (-{H} x)$ is a  linear  deformation of the crossed homomorphism ${H}$. Moreover, this deformation is trivial.
\end{thm}

We need the following lemma to prove this theorem.

\begin{lem}\label{lem:isomorphism}
Let ${H}$ be a crossed homomorphism from $\g$ to $\h$ with respect to the action $\rho$.  Let $\phi_\g:\g\longrightarrow\g$ and $\phi_\h:\h\longrightarrow \h$ be Lie algebra isomorphisms such that Eq.~\eqref{homo-2} holds. Then $\phi_\h^{-1}\circ{H}\circ\phi_\g$ is a crossed homomorphism from $\g$ to $\h$ with respect to the action $\rho$.
\end{lem}

\begin{proof}
 It follows from
straightforward computations.
\end{proof}

{\bf The proof of Theorem \ref{thm:trivial}:}
For any Nijenhuis element $x\in\Nij({H})$, we define
 \begin{eqnarray}\label{trivial-genetator}
 \frkH=d_{{H}} (-{H} x).
 \end{eqnarray}
By the definition of Nijenhuis elements of ${H}$, for any $t$, ${H}_t={H}+t\frkH$ satisfies
\begin{eqnarray*}
        {H}\circ \Big({\Id}_\g+t\ad_x\Big)&=&\Big({\Id}_\h+t\rho(x)\Big)\circ {H}_t,\\
       \Big({\Id}_\h+t\rho(x)\Big)\circ\rho(y)&=&\rho\Big(({\Id}_\g+t\ad_x)(y)\Big)\circ\Big({\Id}_\h+t\rho(x)\Big),\quad\forall y\in\g.
\end{eqnarray*}
 For $t$ sufficiently small, we see that ${\Id}_\g+t\ad_x$ and ${\Id}_\h+t\rho(x)$ are Lie algebra isomorphisms. Thus, we have
$${H}_t=\Big({\Id}_\h+t\rho(x)\Big)^{-1}\circ{H}\circ \Big({\Id}_\g+t\ad_x\Big).$$
By Lemma \ref{lem:isomorphism},
we deduce that ${H}_t$ is a crossed homomorphism from $\g$ to $\h$, for $t$ sufficiently small.
  Thus, $\frkH$
  given by
Eq.~\eqref{trivial-genetator} satisfies the conditions
\eqref{deformation-1} and \eqref{deformation-2}. Therefore,
${H}_t$ is a crossed homomorphism for all
$t$, which means that $\frkH$
given by
Eq.~\eqref{trivial-genetator} generates a deformation. It is straightforward to see  that
this deformation is trivial.\qed

\vspace{2mm}

It is generally not easy to  find  Nijenhuis elements  associated to a crossed homomorphism ${H}$  from a Lie algebra $\g$ to $\h$. Next  we give  examples  on some special  Lie algebras where the Nijenhuis elements
can be explicitly determined.

\begin{ex}{\rm
Let $\g$ be a 2-step nilpotent Lie algebra, i.e.,  $[\g,[\g,\g] ]=0$, and $H:\g\to \g$ a crossed homomorphism with respect to the adjoint action ad of $\g$ on $\g$. It is easy to see  that (\ref{eq:Nij1}), (\ref{eq:Nij2}), (\ref{eq:Nij3}), (\ref{eq:Nij4}) hold for any $x\in\g$. Therefore   $\Nij({H})=\g$ for  any crossed homomorphism $H$ with respect to the adjoint action ad of $\g$ on $\g$.
For example we can take $\g$ to be any Heisenberg algebra.\qed
}
\end{ex}

    \begin{ex}
    \label{ex:2dim}
    {\rm
    Consider the unique $2$-dimensional non-abelian Lie algebra on $\mathbb C^2$. The  Lie bracket is given by
   $[e_1,e_2]=e_1$ for a given basis $\{e_1,e_2\}$.
   For a matrix $\left(\begin{array}{cc}a_{11}&a_{12}\\
a_{21}&a_{22}\end{array}\right)$,
define
$$
   {H} e_1=a_{11}e_1+a_{21}e_2,\quad  {H} e_2=a_{12}e_1+a_{22}e_2.
$$
 ${H}$ is a crossed homomorphism from $\mathbb C^2$ to $\mathbb C^2$ with respect to the adjoint action if and only if
   $$
  {H}[e_1,e_2]=[{H} e_1,e_2]+[e_1,{H} e_2]+[{H} e_1,{H} e_2].
   $$
   By a straightforward computation, we conclude that
  ${H}$ is a crossed homomorphism if and only if
  $
  a_{21}=0, ~(1+a_{11}) a_{22}=0.
  $
So we have the following two cases to consider.

\noindent
(i) If $a_{22}=0$, then we deduce that any ${H}=\left(\begin{array}{cc}a_{11}&a_{12}\\
   0&0\end{array}\right)$ is a crossed homomorphism. In this case, $x=t_1e_1+t_2e_2$ is a Nijenhuis element of ${H}$ if and only if
   $
   t_2(t_1a_{11}+t_2a_{12})=0.
   $
Then for any $t_1\in\mathbb C,$ $t_1e_1$ is a Nijenhuis element for the crossed homomorphism ${H}=\left(\begin{array}{cc}a_{11}&a_{12}\\
   0&0\end{array}\right)$.

\noindent
(ii) If $1+a_{11}=0$, then we deduce that any ${H}=\left(\begin{array}{cc}-1&a_{12}\\
   0&a_{22}\end{array}\right)$ is a crossed homomorphism. In this case, $x=t_1e_1+t_2e_2$ is a Nijenhuis element of ${H}$ if and only if
   $
   t_2(t_2a_{12}-t_1a_{22}-t_1)=0.
   $
   In particular, $e_1+e_2$ is a Nijenhuis element for the crossed homomorphism ${H}=\left(\begin{array}{cc}-1&2\\
   0&1\end{array}\right)$.\qed
   }
    \end{ex}

\begin{ex}{\rm For any crossed homomorphism $H$  from  a finite dimensional semisimple Lie algebra $\g$ over $\C$ to another Lie algebra $\h$  with respect to any action $\rho$, we claim that   $\Nij(H)=0$.

Let $x\in\g$ be a fixed nonzero vector  and assume that $\g_0=[x, \g]$ is abelian, i.e. \eqref{eq:Nij1} holds. We will show that this is impossible.

Denote $n=\dim\g, \g_x=\{y\in\g:[x,y]=0\}$. Considering the linear map $\ad(x):\g\to \g$, we see that $\dim\g_0+\dim\g_x=n$. Let $(\cdot, \cdot)$ be a nondegenerate invariant bilinear form on $\g$. It is easy to see that $\g_x=\g_0^\perp$.
From $0=(0, \g)=([[x, \g], [x, \g]], \g)= ([x, \g],[[x, \g], \g])$ we have
$$[[x, \g], \g]\subset \g_0^\perp=\g_x.$$
We deduce that
$0=[x, [[x, \g], \g]]]=[[x, [x, \g]], \g]$. Since $\g$ is semisimple we see that $[x, [x, \g]]=0$. Thus $x$ is nilpotent.  From Jacobson-Morozov theorem, there are elements $f, h\in \mathfrak{g}$ such that
$$[h, x]=2x, \quad[h, f]=-2f,\quad [x, f]=h.$$
We see that
$[[x, h], [x, f]]= 4x\ne0.$
So $\g_0$ is noncommutative which is a contradiction. Therefore $\Nij(H)=0$.\qed
}
\end{ex}

\section{Conclusion}\label{Sect.6}

We introduce the notions of weak representations of Lie-Rinehart algebras and admissible representations of Leibniz pairs.
By using crossed homomorphisms between Lie algebras, we construct two actions of the monoidal category of representations of Lie algebras   on the category of weak representations  of Lie-Rinehart algebras and the category of admissible representations of Leibniz pairs respectively. In particular, the corresponding bifunctors, called the actions of monoidal categories,   unify and generalize various constructions of modules over certain Cartan type Lie algebras.  New representations of some Lie algebras are also constructed using the actions of monoidal categories. To better understand  crossed homomorphisms and the actions of monoidal categories, we also give a systematic study of deformations and cohomologies of crossed homomorphisms.

There are some natural questions worthy to be considered in the future:
\begin{itemize}
  \item[{\rm(i)}] Whether  the bifunctors $F_H$ and $\huaF_H$  preserve certain properties of representations. For example, when $F_H(V,M) $ and $\huaF_H(V,M)$ are simple if both   $V$ and $M$ are simple?

      \item[{\rm(ii)}] For two crossed homomorphisms $H$ and $H'$, under what conditions the bifunctors $F_H$ and $F_{H'}$ are naturally  isomorphic?

      \item[{\rm(iii)}] How to classify  simple objects in the categories $\WRep_\K(\L)$ and $\ARep_\K(\huaS)$ under certain conditions?
\end{itemize}
\vspace{2mm}
\noindent
{\bf Acknowledgements. } This research is supported by NSFC (11871190,11971315,11922110) and NSERC (311907-2020). We would like to thank L. Guo and S. Ng for help discussions on the first version of the paper.


 \end{document}